\documentclass[11pt]{article} %{amsart}
\usepackage{amssymb}
\usepackage{amsthm,amsmath,amssymb}
\usepackage{amsthm,amsmath}

\usepackage{amscd, amssymb}
\usepackage{hyperref}
\usepackage{comment}
\usepackage{enumitem}
\setlist[enumerate,1]{label={(\roman*)}}
\newtheorem{theorem}{Theorem}[section]
\newtheorem{lemma}[theorem]{Lemma}
\newtheorem{prop}[theorem]{Proposition}
\newtheorem{corollary}[theorem]{Corollary}
\newtheorem{conjecture}[theorem]{{Conjecture}}

\newtheorem{claim}[theorem]{{Claim}}

\theoremstyle{remark}
\newtheorem{remark}[theorem]{Remark}
\newtheorem{rem}[theorem]{Remark}

\theoremstyle{definition}

\newtheorem{ex}[theorem]{{Example}}
\newtheorem{definition}[theorem]{{Definition}}
\newtheorem{df}[theorem]{{Definition}}

\def\bclaim{\begin{claim}}
\def\eclaim{\end{claim}}
\def\bdefin{\begin{definition}}
\def\edefin{\end{definition}}
\def\bcor{\begin{corollary}}
\def\ecor{\end{corollary}}
\def\bthm{\begin{theorem}}
\def\ethm{\end{theorem}}
\def\bconj{\begin{conjecture}}
\def\econj{\end{conjecture}}
\def\blem{\begin{lemma}}
\def\elem{\end{lemma}}
\def\blemma{\begin{lemma}}
\def\elemma{\end{lemma}}
\def\bprop{\begin{prop}}
\def\eprop{\end{prop}}
\def\bremark{\begin{remark}}
\def\eremark{\end{remark}}

\newcommand{\R}{\mathbb R}
\newcommand{\RR}{\mathbb R}
\newcommand{\C}{\mathbb C}
\newcommand{\CC}{\mathbb C}
\newcommand{\N}{\mathbb N}
\newcommand{\Z}{\mathbb Z}

\renewcommand{\P}{\vec P}
\newcommand{\s}{\mathfrak s}

\DeclareMathOperator{\id}{Id}
\DeclareMathOperator{\spec}{spec}
\DeclareMathOperator{\ham}{Ham}
\def\sign{\h{\rm sign}} %\DeclareMathOperator{\sign}{sign}

\newcommand{\calF}{\mathcal F}
\newcommand{\calD}{\mathcal D}
\newcommand{\calC}{\mathcal C}
\newcommand{\la}{\left\langle}
\newcommand{\ra}{\right\rangle}
\newcommand{\tJ}{\hat J}
\newcommand{\tom}{\widehat \omega}
\newcommand{\tO}{\widehat \Omega}
\newcommand{\tg}{\hat g}

\def\USC{\operatorname{USC}}
\def\usc{\operatorname{usc}}
\def\diag{\operatorname{diag}}
\def\topp{\operatorname{top}}
\def\bott{\operatorname{bot}}
\def\Im{\operatorname{Im}\,}
\def\im{\operatorname{Im}\,}
\def\Re{\operatorname{Re}\,}
\def\Sym{{\hbox{\rm Sym}^2}}
\def\Int{\hbox{\rm int}\,}
\def\cl{\hbox{cl}}

\def\Im{{\operatorname{Im}\,}}
\def\Re{{\operatorname{Re}\,}}

\def\max{{\operatorname{max}}}

\def\tr{\hbox{\rm tr}}

\def\q{\quad}
\def\qq{\qquad}

\def\bpf{\begin{proof}}
\def\epf{\end{proof}}
\def\beq{\begin{equation}}
\def\eeq{\end{equation}}
\def\beqno{\begin{equation*}}
\def\eeqno{\end{equation*}}
\def\eaeq{\end{aligned}}
\def\baeq{\begin{aligned}}

\def\P{{\mathcal P}}
\def\tP{\widetilde{{\mathcal P}}}
\def\SA{\hbox{\rm SA}}
\def\tF{\widetilde F}
\def\del{\partial}
\def\sm{\setminus}
\def\nt{\nabla^2}
\def\HL{Harvey--Lawson }
\def\HLno{Harvey--Lawson}
\def\O{X}
\def\D{D}
\def\vp{\varphi}
\def\FDP{$F$-Dirichlet problem }
\def\vF{\vec F}
\def\vtF{\vec{\tF}}
\def\vcalF{\vec \calF}
\def\vtcalF{\vec{\widetilde{\calF}}}
\def\tcalF{{\widetilde{\calF}}}
\def\n{\nabla}
\def\h#1{\hbox{#1}}
\def\MA{Monge--Amp\`ere }
\def\II{\hbox{\rm II}}
\def\wt{\widetilde}
\def\lb{\label}

\def\th{\theta}
\def\MA{Monge--Amp\`ere }
\def\MAno{Monge--Amp\`ere}
\def\K{K\"ahler }
\def\i{\sqrt{-1}}
\def\del{\partial}

\def\ra{\rightarrow}
\def\eps{\epsilon}
\def\del{\partial}

\newcommand{\calO}{\mathcal{O}}
\def\sm{\setminus}

\def\O{\Omega}
\def\sm{\setminus}
\def\w{\wedge}
\def\o{\omega}

\def\vp{\varphi}
\def\beq{\begin{equation}}
\def\eeq{\end{equation}}
\def\bi#1{\bibitem{#1}}

\def\h#1{\hbox{#1}}

\def\calO{{\mathcal O}}
\def\calL{{\mathcal L}}
\def\calM{{\mathcal M}}
\def\calP{{\mathcal P}}

\def\calE{{\mathcal E}}
\def\calS{{\mathcal S}}

\def\La{\Lambda}
\def\la{\lambda}
\def\Th{\Theta}

\def\graph{{\operatorname{graph}}}

\def\dist{\h{\rm dist}}

\def\Symn{\Sym(\RR^n)}
\def\Symnplusone{\Sym(\RR^{n+1})}
\def\wtTh{\widetilde{\Th}}
\def\wtThk{\widetilde{\Th}_k}
\def\dx{{\bf dx}}

\renewcommand{\S}{Section~}

\begin{document}
\title{The degenerate special Lagrangian equation}
\author{Yanir A. Rubinstein and Jake P. Solomon}
\date{June 2015}

\maketitle

\begin{abstract}
This article introduces the {\it degenerate special Lagrangian equation (DSL)}
and develops the basic analytic tools to construct and study its solutions.
The DSL governs geodesics in the space of positive graph Lagrangians in $\mathbb{C}^n.$
Existence of geodesics in the space of positive Lagrangians is an important step in a program for proving existence and uniqueness of special Lagrangians.
Moreover, it would imply certain cases of the strong Arnold conjecture from Hamiltonian dynamics.

We show the DSL is degenerate elliptic. We introduce a {\it space-time Lagrangian angle} for one-parameter families of graph Lagrangians, and construct its regularized lift. The superlevel sets of the regularized lift define subequations for the DSL in the sense of Harvey--Lawson. We extend the existence theory of Harvey--Lawson for subequations to the setting of domains with corners, and thus obtain solutions to the Dirichlet problem for the DSL in all branches. Moreover, we introduce the {\it calibration measure}, which plays a r\^ole similar to that of the Monge--Amp\`ere measure in convex and complex geometry. The existence of this measure and regularity estimates allow us to prove that the solutions we obtain in the outer branches of the DSL have a well-defined length in the space of positive Lagrangians.
\end{abstract}

%% AMS 2010 Classification:
%%
%% Primary:
%% 53D12  Lagrangian submanifolds; Maslov index
%% Secondary:
%% 53C38  Calibrations and calibrated geometries
%% 35D40  Viscosity solutions
%% 35J66  Nonlinear boundary value problems for nonlinear elliptic equations

%\pagestyle{plain}

\tableofcontents

\section{Introduction}
\subsection{The DSL}
Let $D \subset \R^n$ be a bounded domain with smooth boundary $\partial D$ and let
\beq
\label{calDEq}
\calD:= (0,1) \times D \subset \R^{n+1}.
\eeq
The coordinate on $(0,1)$ is called $t$ and the coordinates on $D$ are called $x.$
Denote by $I_n$ the diagonal $(n+1) \times (n+1)$ matrix with diagonal entries $(0,1,\ldots,1).$
We say a function $u\in C^2(\calD)$ satisfies the {\it degenerate special Lagrangian (DSL)} equation of phase $\th\in(-\pi,\pi]$ if
\beq
\label{DSLMainEq}
\Im \left( e^{-\i \th} \det (I_n + \i\nabla^2 u)\right) = 0, \qquad
\Re \left(e^{-\i \th} \det \left(I + \i\nabla_x^2 u\right)\right) > 0.
\eeq
The goal of this article is to study the Dirichlet problem for the DSL equation.

We prove that the DSL is degenerate elliptic. More generally, the relationship between the DSL and the special Lagrangian equation, introduced in the classical work of Harvey--Lawson~\cite{HL82}, is analogous to the relationship between the homogeneous and inhomogeneous \MA equations. Thus, it is natural from the analytic point of view to study the DSL and the equation has a rich structure. Yet the precise formulations and proofs of many properties of the DSL are surprisingly complex in comparison with their \MA analogues.

From the geometric point of view, the DSL governs geodesics in the space of positive Lagrangians. Such geodesics play a crucial r\^ole in the program of the second author~\cite{S1,S2} concerning existence and uniqueness of special Lagrangians in Calabi--Yau manifolds, the geometry of the space of positive Lagrangians, and stability conditions for Lagrangian submanifolds in the context of mirror symmetry. Another geometric motivation for this work is the observation of Lemma \ref{ArnoldLemma} that whenever a pair of positive Lagrangians is connected by a sufficiently regular geodesic, the number of intersection points is bounded below by the number of critical points of a function on one of them. Thus, this article can be viewed as a first step in a new approach to the strong Arnold conjecture~\cite{Ar89}.  Lemma~2.1 applies equally to tranverse and non-transverse Lagrangians. We refer the reader to Section~\ref{LeviCivitaSubSec} for a more in depth discussion.

Previous work of the second author and Yuval~\cite{SY} constructed geodesics in the space of positive Lagrangians in Milnor fibers using $O(n)$ symmetry to reduce the
problem
%equation
to a Hamiltonian flow ODE. The present work constructs geodesics of positive Lagrangians in $\C^n$ in the absence of any symmetry assumptions using the theory of fully non-linear degenerate elliptic PDE. Unlike in the case of the non-degenerate special Lagrangian equation studied by Harvey--Lawson \cite[Corollary 2.14]{HL82}, we cannot use the implicit function theorem to construct many solutions for the DSL because the symbol is degenerate.

\subsection{Results}
We now give an overview of our main results mostly avoiding technical background. The reader is referred to later sections for sharper statements.

An understanding of the DSL starts with
establishing a notion of subsolutions for the DSL.
With an eye toward \HLno's Dirichlet duality theory
\cite{HL},
we are led to construct a {\it subequation}
for the DSL. A subequation, known also as a Dirichlet set~\cite{HL}, is a proper closed subset
$F$ of the set of symmetric matrices $\Sym(\RR^m)$
that satisfies
\beq
\lb{DirichletSetPropEq}
F+\P\subset F,
\eeq
where $\P\subset\Sym(\RR^m)$ is the set of nonnegative matrices.
Such a set $F$ is a subequation for a PDE
of the form $f(\nabla^2u(x))=0$ for functions $u \in C^2(U),\,U\subset\RR^m,$ if $C^2(U)$ solutions
of the equation satisfy $\nabla^2u(x)\in\del F$
for each $x\in U$. A subequation $F$ gives rise to a natural notion of subsolution. Namely, $u \in C^2(U)$ is a subsolution if $\nabla^2u(x) \in F$ for all $x \in U.$ Moreover, $F$ gives rise to a weak version of the Dirichlet problem for each domain $U \subset \RR^m.$ \HL show existence and uniqueness of continuous solutions to the $F$-Dirichlet problem under certain assumptions on the boundary of $U.$ If the solution is in $C^2(U)$, it must be a solution in the classical sense.

To obtain a subequation for the DSL, we associate to each $u \in C^2(\calD)$ the circle valued function
\[
\Theta_u(t,x) = \Theta(\nabla^2 u(t,x)) = \arg \det(I_n + \i \nabla^2 u(t,x)) \in S^1,
\]
defined where $\det(I_n + \i \nabla^2 u(t,x)) \neq 0.$ We call $\Theta$ the \emph{space-time Lagrangian angle} by analogy with the Lagrangian angle of \HL\cite{HL82}. If $u$ solves the DSL of phase~$\theta,$ then $\Theta_u \equiv \theta.$
First, we promote this equality of angles to an equality of real numbers. Then, the subequation and its corresponding notion of subsolution are obtained by weakening the equality to an inequality using the order of $\RR.$
Indeed, let $\calS \subset \Symnplusone$ be the set of matrices such that the first column and row vanish identically. For $B$ a complex matrix, denote by $\spec(B)$ the set of its eigenvalues, and for $\lambda \in \spec(B)$ denote by $m(\lambda)$ its multiplicity as a root of the characteristic polynomial. Consider the branch of $\arg$ with values in $(-\pi,\pi].$ For $A \in \Symnplusone\sm\calS,$ define
\[
\widehat \Theta(A) = \sum_{\lambda \in \spec(I_n + \i A)} m(\lambda)\arg(\lambda).
\]
So, $\arg \det(I_n + \i A) = \widehat\Theta(A) \mod 2\pi.$
\bthm
The function $\widehat \Theta$ is well-defined and differentiable on $\Symnplusone\sm\calS$. Denote by $\widetilde \Theta$ the minimal upper semi-continuous extension of $\widehat \Theta$ to $\Symnplusone.$ Then, for each $c\in(-(n+1)\pi/2,(n+1)\pi/2)$ such
that $c\equiv \th\, \mod \, 2\pi$, the set
$$
\calF_c
=
\{
A\in\Symnplusone\,:\,
\wtTh(A)\ge c\}
$$
is a subequation for the DSL of phase $\th$.
\ethm

This result is contained in Theorems \ref{USCThm} and~\ref{calFcThm} and Corollary~\ref{CtwoDSLCor} below. The different choices of $c$ for a given $\theta$ correspond to the branches of the DSL. An interesting feature of the DSL, not seen in \MAno, is the locus where $\det(I_n + \i \nabla^2 u(t,x)) = 0$ and thus $\Theta_u$ is not defined. We show this is precisely the critical locus of $\del_t u$ or, equivalently, the locus where $\nabla^2u \in \calS.$ The spacetime Lagrangian angle $\Theta$ and its lift $\widehat \Theta$ cannot be extended continuously over this locus. So, we are forced to consider instead the minimal upper semi-continuous extension $\wtTh.$ It is a beautiful feature of the DSL equation and \HLno's theory that $\wtTh$ nonetheless gives rise to a subequation. Other subequations we are aware of arise as superlevel sets of \emph{continuous} functions.

\HL prove the existence of continuous solutions to the Dirichlet problem for general subequations on domains with smooth boundary that is ``strictly convex'' in an appropriate sense. However, our domain $\calD = (0,1)\times D$ has corners and is not ``strictly convex.'' In Theorem~\ref{tm:subs}, we extend \HLno's results to a class of domains including $\calD.$ Possible applications
go beyond the DSL. For instance, Theorem \ref{tm:subs}
allows one to show that the homogeneous
real/complex \MA equation on certain product
domains has continuous solutions in all
branches. Previously, the only solutions
known to exist were in the convex/psh or concave/plurisuperharmonic branches.

Building on the general existence result of Theorem~\ref{tm:subs},
we prove the existence and uniqueness of solutions for all branches of the $\calF_c$-Dirichlet problem and hence for the endpoint problem for geodesics. A special case of our result is the following theorem.
\bthm
\label{MainExistenceThm}
Let $D\subset \RR^n$ be a bounded strictly convex domain. For $i = 0,1,$ let $\vp_i\in C^2\left(D\right)$ satisfy
\begin{equation}\label{eq:plbc}
\tr\tan^{-1}(\nabla^2\vp_i)
\in(c-\pi/2,c+\pi/2).
\end{equation}
There exists a unique solution $u \in C^0(\overline\calD)$ for the $\calF_c$-Dirichlet problem with $u|_{\{i\} \times D}= \vp_i$ and $u|_{[0,1]\times \del D}$ affine in $t.$ Moreover, $u$ is Lipschitz in $t$ on $\overline\calD.$ If $|c|\in [n\pi/2,(n+1)\pi/2)$, then $u \in C^{0,1}(\calD)$.
\ethm

Condition~\eqref{eq:plbc} is equivalent to the geometric condition that the graph of $\nabla \vp_i$ is a positive Lagrangian.
For more detailed statements, we refer to Theorem~\ref{tm:cgeq} and Lemma~\ref{tm:Lip}. Remark~\ref{rm:convexity} explains how to deduce Theorem~\ref{MainExistenceThm} from Theorem~\ref{tm:cgeq}.

The last statement in Theorem \ref{MainExistenceThm} establishes further regularity
for solutions in the outermost branches.
The subequations $\calF_c$ with $|c|\in [n\pi/2,(n+1)\pi/2)$
are the analogues  of the convex/concave and
plurisubharmonic/plurisuperharmonc
branches in the study of the real and complex \MA equations.
For Monge--Amp\`ere, essentially the only regularity results beyond
$C^0$ to date concern these branches. Thus,
Theorem \ref{MainExistenceThm} can be considered as
giving essentially optimal
regularity for all the inner branches.
For the outermost branches of the equation, somewhat stronger results
are possible using completely different PDE techniques that do not
work for the other branches. We
leave such a treatment to a separate article.

For a solution $u$ of the DSL in one of the outermost
branches, Theorem~\ref{CalibThm} shows that the restriction of
$\Re (dz_1 \wedge \ldots dz_n)$
to the graph of $\nabla_x u|_{\{t\}\times \R^n}$ in $\C^n = \R^n \oplus \R^n$ is a well-defined positive measure, which we call the {\it calibration measure}. This result holds despite the fact that such graphs may not have a tangent space at every point.
In addition, Lemma~\ref{tm:Lip} gives a partial Lipschitz a priori estimate for the solution $u$. Combining the Lipschitz estimate and the existence of the calibration measure, Theorem~\ref{lengthThm} shows the length of the geodesic of graph Lagrangians corresponding to $u$ is well-defined. Furthermore, integrating the calibration measure along paths of Lagrangians, we obtain the {\it calibration functional}. This functional is affine along smooth geodesics, and we conjecture it is affine along weak geodesics as well. Thus, the calibration measure plays a r\^ole in the geometry of positive Lagrangians similar to that of the \MA measure in convex and complex geometry.

\subsection{Organization}
The article is organized as follows. In \S\ref{SPLSec} we
recall the Riemannian metric on the space of positive
Lagrangians introduced in \cite{S1} along with the associated notions of parallel transport and geodesics.
Lemma~\ref{ArnoldLemma} shows that a version of the strong Arnold conjecture
follows
if we assume
the existence of sufficiently regular geodesics.
Finally, Proposition~\ref{SPLgeodEq} shows that the geodesic equation for the space of positive graph
Lagrangians in $\C^n$ is equivalent to the DSL.

In Section~\ref{DDSec} we establish the basic properties of the regularized lift of the space-time Lagrangian angle $\wtTh$.
In Section \ref{DegEllipSec} we compute the symbol of the linearization
of the DSL at a solution and prove the DSL is degenerate elliptic.
In Section \ref{DSetSec} we construct subequations
associated to the DSL in the sense of Harvey--Lawson~\cite{HL}.
This is the key to the definition of the weak solutions
of the DSL that are the main focus of the remainder of the article.

Section~\ref{DDTheoreySec} recalls the main features
of the Dirichlet duality theory of Harvey--Lawson.
Section~\ref{DDweakSec} extends Dirichlet duality theory to include weaker boundary assumptions. Section~\ref{MainThmSec} applies the results of Section~\ref{DDweakSec} to prove existence and uniqueness of solutions of the DSL in all branches.

Section \ref{RegSec} establishes basic regularity results for solutions of the DSL.
First, \S\ref{ConvFibSubSec} shows that solutions to the $\calF_c$-Dirichlet problem
are ``convex in the $x$ variables'' in the sense of the subequation
for the nondegenerate special Lagrangian equation.
Second, \S\ref{LipschitzSubSec} shows that solutions to the $\calF_c$-Dirichlet problem are Lipschitz in the variable $t$.
Section~\ref{CalibSec} introduces the {\it calibration measure}
for subsolutions in the outermost branches of the DSL.
Section~\ref{ssec:length} shows that the Riemannian length functional is well-defined on the solutions we construct for the DSL in the outermost branches.
In Section~\ref{ssec:cal} we introduce the {\it calibration
functional}
and show it is affine along
sufficiently regular geodesics. Furthermore, we formulate a conjecture characterizing weak solutions of the DSL as those subsolutions along which the calibration functional is affine. Appendix~\ref{LimitDSet} proves an alternative formula for the lifted space-time Lagrangian angle $\wtTh$ (Corollary~\ref{wtThkFormulaCor}) by viewing the DSL as a limit of non-degenerate special Lagrangian equations. Appendix~\ref{sec:gdsl} gives a geometric formulation of the DSL equation valid in an arbitrary Calabi-Yau manifold.

\section{The space of positive Lagrangians}
\label{SPLSec}
In this section we review the construction  of a weak Riemannian
metric on the space of positive Lagrangians \cite[\S5]{S1}. We then
formulate the equation for geodesics
in this space in the case $X=\C^n$, introducing the {\it degenerate special Lagrangian equation}.
% that plays
% a central r\^ole in this article.

\subsection{Lagrangians in a Calabi--Yau manifold}

\def\dz{{\bf dz}}
%\cite[Definition 5.1] {S1}.

Let $(X, J,\omega,\Omega) $ be a Calabi--Yau manifold of complex dimension $n$.
This amounts to $(X, J,\omega)$ being complex, where $J$ denotes the complex structure,
so $g(\,\cdot\,,\,\cdot\,):=\o(\,\cdot\,,J\,\cdot\,)$ is a \K metric,
and $\Omega$ is a holomorphic nowhere vanishing $(n,0)$-form,
with  $\i^{n^2}\Omega\wedge\overline{\Omega}=(2\o)^n/n!$.
Thus, in a local coordinate chart $U\ni p$, there are holomorphic coordinates $z=(z^1,\ldots z^n)$
such that
$\omega(p)=\frac{\i}2\sum_j dz^j\w d\bar{z}^j$
and $\Omega(p) =\dz:=dz^1\wedge \cdots \wedge dz^n$. For any tangent $n$-plane $\tau\in T_pX$
\cite[p. 88]{HL82},
$$
|\dz (\tau)|^2=|\tau\wedge J\tau|_{g(p)}\le
|\tau|_{g(p)}^2,
$$
%where the norms are those induced by $\omega(p)$,
with equality if and only if $\tau$ is Lagrangian.
Thus, if $\Lambda\subset X $ is an $n$-dimensional submanifold, then globally on $\Lambda$,
\beq\label{HLInEq}
|\Omega|_\Lambda|^2=
|\Re \Omega|_\Lambda|^2+|\Im \Omega|_\Lambda|^2\le 1,
\eeq
where the norms are those induced by $g$,
with equality if and only if $\Lambda$ is Lagrangian.
%Of course, \eqref{HLInEq} implies
%$$
%\big|\Omega|_\Lambda\big|^2=
%|\Re e^{-\i \th}\Omega|_\Lambda|^2+|\Im e^{-\i c}\Omega|_\Lambda|^2\le 1,
%$$
%for any $\th\in(-\pi/2,\pi/2]$.
From now and on, $\Lambda$ will always denote a Lagrangian submanifold.
In particular,
\beq
\label{AngleEq}
\Omega|_\Lambda= e^{\i\theta_\Lambda} dV_g|_\Lambda,
\eeq
for some function $\theta_\Lambda:  \Lambda\rightarrow S^1$,
where $dV_g$ equals the Riemannian volume form associated to $g$.
Following Harvey--Lawson, $\Lambda$ is called {\it special Lagrangian (SL) of phase $\theta$} if
$\theta_\Lambda$ is
constant and equal to $\th\in(-\pi,\pi]$.
In other words,
$\Lambda$ is calibrated by
$\Re (e^{-\i \th}\Omega)$.
Alternatively, up to a choice of orientation, the phase $\theta$ special Lagrangian condition is equivalent to
$\Im (e^{-\i \th}\Omega|_\La)=0$ \cite{HL82}.

\subsection{The space of positive Lagrangians}
\lb{SPLSubSec}
Let $L$ be a possibly non-compact connected $n$-dimensional manifold.
Define
$$
\baeq
\calL &=
\left \{
\Gamma \subset X \h{\ an oriented Lagrangian }
\h{submanifold diffeomorphic to } L
\right \}.
\eaeq
$$
%
%
% $$
% \baeq
% \calL_\th &= \left \{ \begin{array}{l}
% \Gamma \subset X \h{\ an oriented Lagrangian } \\
% \h{submanifold diffeomorphic to } L
% \end{array}
% \left |
% \begin{array}{ll}
% \Im\int_\Gamma e^{-\i \th}\Omega = 0 & \text{if $\Gamma$ is compact,} \\
% \Im (e^{-\i \th}\Omega|_\Gamma) = 0 &  \text{outside a compact} \\
% & \text{$K \subset \Gamma$, otherwise}
% \end{array}
% \right.
% \right \}.
% \eaeq
% $$
For $\th\in(-\pi,\pi]$, the space of $\th$-\emph{positive} Lagrangians is defined as
\[
\calL_\th^+ = \{\Gamma \in \calL : \Re (e^{-\i \th}\Omega)|_\Gamma > 0 \}.
\]
%Following \cite[\S5.3]{S1}, we call
Note that $\Gamma$ is $\th$-positive iff $|\theta_\Gamma-\th|<\pi/2$.
%where we identify $S^1$ with $[-\pi,\pi)$.
In other words, $\Re (e^{-\i \th}\Omega)$ restricts to a volume form on  $\Gamma$.
This notion (with $\th=0$) was used by Wang in a different context
\cite[p. 302]{Wang}.
In particular,
any special Lagrangian  in $\calL$
of phase $\theta$
is contained
in $\calL_{\th'}^+$
for each $\th'\in(\th-\pi/2,\th+\pi/2)$.

Denote by $\ham (X,\omega)$ the group of compactly supported Hamiltonian diffeomorphisms of $X.$ Denote by $\calO_\th\subset\calL^+_\th$ a connected component of the intersection of $\calL^+_\th$ with an orbit of $\ham (X,\omega)$ acting on $\calL.$ The space $\calO_\th$ is called an exact isotopy class.

We now describe the tangent space to $\calO_\th$ at $\Gamma\in\calO_\th$.
Recall that whenever $f:X\ra Y$ is a smooth map, $v$ is a vector
field along $f$, and $\alpha$ is a
differential $k$-form on $Y$, we define $\iota_v\alpha$ to be the $(k-1)$-form on $X$ satisfying
\beq\label{IntProdEq}
\iota_v\alpha(X_1,\ldots,X_{k-1}):=
\alpha(v,df(X_1),\ldots,df(X_{k-1})).
\eeq
Given a
 curve $\Lambda:(-\eps,\eps)\ra\calO_\th$ with $\Lambda(0) = \Gamma,$
we choose a smooth family of
diffeomorphisms $g_t:L\ra\Lambda_t : = \Lambda(t)$,
and consider the 1-form
$
\iota_{dg/dt}\o.
$
Since $\Lambda_t$ are Lagrangian, this 1-form is closed.
By Akveld--Salamon \cite[Lemma 2.2]{AS}, it is also exact.
If $L$ is non-compact, let $h_t:\Lambda_t\ra\RR$ be the unique compactly supported
function such that
\beq
\label{dgdtEq}
\iota_{dg_t/dt}\o=d(h_t\circ g_t).
\eeq
If $L$ is compact, let $h_t$ be the unique function satisfying~\eqref{dgdtEq} and
\beq\label{nmlz}
\int_{\Lambda_t} h_t \Re\Omega = 0.
\eeq
According to Akveld--Salamon \cite[Lemma 2.1]{AS},
$h_t$ is independent
of the choice of the diffeomorphisms $g_t$. Thus, in either case, we
make the identification $d\La_t/dt \equiv h_t$. If $L$ is non-compact, this
identifies the tangent space of $\calO_\th$ at $\Gamma$ with the space of compactly supported smooth functions on $\Gamma$,
$$
T_\Gamma\calO_\th\simeq C_0^\infty(\Gamma).
$$
If $L$ is compact, this identifies $T_\Gamma\calO_\th$ with the space of smooth functions satisfying the normalization condition~\eqref{nmlz}.
Following \cite{S1}, we define a weak Riemannian metric on $\calO_\th$ by
$$
(h, k)_\th|_\Gamma: =\int^{}_{\Gamma} hk\Re (e^{-\i \th}\Omega|_\Gamma), \q h,k\in T_\Gamma\calO_\th, \q \Gamma\in \calO_\th.
$$

\subsection{The Levi--Civita connection, geodesics, and the Arnold conjecture}

\label{LeviCivitaSubSec}

Let $\La:[0,1]\ra \calO_\th$ be a path in $\calO_\th$, and write $\Lambda_t = \Lambda(t).$ Denote by
$g_t:L\ra \La_t$ a one-parameter family of diffeomorphisms.
Let $h_t\in T_{\Lambda_t}\calO_\th$ be a vector field on $\calO_\th$
along $\La$.  In~\cite[Section 4]{S2}, it is shown that the Levi-Civita covariant derivative of $h_t$ in the direction $d\La_t/dt$ is defined by
\begin{equation}\label{eq:lcc}
\frac{Dh_t}{dt} =
\Big(\frac{\del}{\del t}(h_t\circ g_t)+g_t^* dh_t(w_t)\Big)\circ g_t^{-1},
\end{equation}
where $w_t\in\Gamma(L,TL)$ is defined as the unique solution of
\begin{equation}\label{eq:crfl}
\iota_{w_t}g_t^*\Re\big(e^{-\i \th}\O\big)
=
-\iota_{dg_t/dt}\Re\big(e^{-\i \th}\O\big).
\end{equation}
In particular, expression~\eqref{eq:lcc} is independent of the choice of $g_t.$

Another way to think of the covariant derivative is the following. Let $\phi_t : L \to L$ be a family of diffeomorphisms such that
\[
\frac{d\phi_t}{dt} = w_t\circ \phi_t.
\]
Let $\tilde g_t = g_t \circ \phi_t.$ Then
\begin{equation}\label{eq:iotg}
\iota_{d\tilde g_t/dt}\Re\big(e^{-\i \th}\O\big) = 0,
\end{equation}
and consequently
\beq
\label{ReformulateGeodEq}
\frac{Dh_t}{dt} = \frac{\partial}{\partial t} (h_t \circ \tilde g_t).
\eeq

As usual, $\Lambda$ is a geodesic if for $h_t = \frac{d\Lambda_t}{dt}$ we have $\frac{Dh_t}{dt} = 0.$ In other words, $h = h_t \circ \tilde g_t : L \to \R$ is independent of $t.$ Let $p$ be a critical point of $h.$ Combining equation~\eqref{eq:iotg} and equation~\eqref{dgdtEq} with $\tilde g_t$ in place of $g_t$, we conclude that $\frac{d\tilde g_t}{dt}(p) = 0$ for all $t.$ Thus, we obtain the following, which is interesting primarily when $L$ is compact.
\begin{lemma}
\label{ArnoldLemma}
If $\Lambda_0,\Lambda_1 \in \calO_\th$ are joined by a geodesic, then
\[
\#(\Lambda_0 \cap \Lambda_1) \geq \#Crit(h).
\]
Here, $\#$ denotes the unsigned cardinality of a set.
\end{lemma}

The lemma links the existence of geodesics to the original version of Arnold's influential conjecture on fixed points of Hamiltonian symplectomorphisms~\cite[Appendix 9]{Ar89}.
\begin{conjecture} {\rm (Arnold)}\label{conj:Ar}
Every Hamiltonian symplectomorphism $\phi$ of a compact symplectic manifold $(M,\omega_M)$ has at least as many fixed points as a smooth function $H: M\to \R$ has critical points.
\end{conjecture}
This conjecture can be rephrased in terms of Lagrangian intersections. Indeed, consider $X = M \times M$ with projections $p_1,p_2 : X \to M$ and $\omega = -p_1^*\omega_M + p_2^*\omega_M.$ Take $\Lambda_0$ to be the diagonal, and $\Lambda_1 = (\id \times \phi)(\Lambda_0).$ Then $\# \Lambda_0 \cap \Lambda_1$ is the number of fixed points of $\phi.$ If $M$ is Calabi-Yau of dimension $m$ with complex structure $J_M$ and holomorphic $(m,0)$-form $\Omega_M,$ we equip $X$ with the complex structure $-J_M\oplus J_M$ and holomorphic $(n,0)$-form $p_1^*\overline\Omega_M \wedge p_2^*\Omega_M.$ Then the diagonal $\Lambda_0$ is a special Lagrangian and thus positive. Positivity of $\Lambda_1$ translates to a subtle condition on $\phi,$ which certainly holds if $\phi$ is $C^1$-close to the identity, but does not imply $\phi$ is close to the identity.

Arnold's Conjecture~\ref{conj:Ar} can be interpreted in two ways. First, we can assume $\phi$ has non-degenerate fixed points, and compare the number of fixed points with the critical points of a Morse function on $M.$ Second, we can consider $\phi$ with possibly degenerate critical points, and compare the number of fixed points with critical points of an arbitrary function on $M.$ In the first case, the conjectured lower bound is larger, but the second case is more general.

Currently, most results on the Arnold conjecture concern one of two weak versions. The main tool is Floer homology. In the first weak version~\cite{Fl88,Fl89,Fu99,HS95,LT98}, the symplectomorphism $\phi$ is assumed to have non-degenerate fixed points and the number of critical points of $H$ is
replaced with $\sum_{i=0}^n\dim H_i(M).$ In the second weak version~\cite{Fl89,Fl89a,Ho88,LO96},
the symplectomorphism $\phi$ may have degenerate fixed points, but the number of critical points of $H$ must
be replaced with the cup-length of $M.$ Recent work~\cite{Ba14,OP14} relates the number of fixed points of $\phi$ to $\pi_1(M)$ if the fixed points of $\phi$ are non-degenerate. Under certain assumptions, the second version of the original conjecture has been proven by Rudyak~\cite{Ru99}.
By Lemma~\ref{ArnoldLemma}, existence of geodesics would yield results on both versions of the original conjecture.

The problem of Lagrangian intersections has also been considered widely starting with Arnold himself~\cite{Ar89}. Floer's first paper on Floer homology~\cite{Fl88} concerned the Lagrangian version of Arnold's conjecture. However, the general Lagrangian intersection problem is considerably more subtle as $J$-holomorphic disks with boundary in the Lagrangian give rise to obstructions to defining Floer homology~\cite{Fu09}. Moreover, even if Lagrangian Floer homology is defined, it may not be isomorphic to the singular homology of the Lagrangian.

\subsection{Geodesics of graph Lagrangians}\label{sec:graphs}

Consider $X=\C^n$ with the standard Euclidean symplectic
form
\[
\o=\frac{\i}{2} \sum_j dz_j\w d\bar z_j= \sum_j dx_j\w dy_j.
\]
Identify $\C^n$ with
$\RR^n\oplus \i \RR^n$ and $L$ with $\RR^n\times\{0\}\subset \C^n$.
Consider a path $\Lambda$ of Lagrangian graphs of the form
$\La_t=\h{graph}(d_xk(t,\,\cdot\,))$ for $k\in C^2([0,1]\times\RR^n)$ constant in $t$ outside a compact set. Denote by $h_t$ the vector field along $\Lambda$ given by $h_t = d\La_t/dt.$ Take
$g_t(x)=(x,d_xk(t,x))$, so
$$
\frac{dg_t}{dt}=\sum_{i=1}^n\frac{\del^2 k}{\del t \del x_i}\frac{\del}{\del y_i}\Big|_{g_t(x)}.
$$
Thus $\iota_{dg_t/dt}\o=-d_x\dot k(t,\cdot)$, where
$\dot k=\del_t k$, and by \eqref{dgdtEq} the vector field $h_t$ is given by
\[
h_t\circ g_t(x)=-\dot k(t,x).
\]

Recalling the definition of the interior product along a map
\eqref{IntProdEq},
$$
\baeq
\iota_{dg_t/dt}\O
&=
\sum_{i=1}^n
\O\bigg(dg_t/dt, dg_t\Big(\frac{\del}{\del x_1}\w\cdots
\w\widehat{\frac{\del}{\del x_i}}\w\cdots\w \frac{\del}{\del x_n}\Big)\bigg)
dx_1\w\cdots \w\widehat{dx_i}\w\cdots\w dx_n
\cr
&=
\sum_{i=1}^n
\det B_i
dx_1\w\cdots \w\widehat{dx_i}\w\cdots\w dx_n,
\eaeq
$$
where $B_i, i=0,\ldots,n$, is the $n$-by-$n$ matrix obtained by removing the
$(i+1)$-th column from the $n$-by-$(n+1)$ matrix
$$
B=
\left(\begin{array}{c|ccc}
& & &\cr
\sqrt{-1}\partial_t \nabla_x k & &I+\i\nabla^2_xk&\cr
&&&
\end{array}\right).
$$
Next, denoting $\del_x k=(k_1,\ldots,k_n)$, we have
\begin{align}
\label{CndzEq}
g_t^*\O &=d(x+\i k_1)\w\cdots\w d(x+\i k_n)\\ &=\det[I+\i\nabla^2_xk] \;dx_1\w\cdots\w dx_n, \notag\\
&=\det B_0 \; dx_1\w\cdots\w dx_n. \notag
\end{align}
As $\La_t\in\calO_\th$, we have $\Re \big(e^{-\i \th}\det B_0\big)>0$. Now, set $w_t = \sum_{i=1}^n a^i(t,x)\frac{\del}{\del x_i}$. So,
$$
\baeq
\iota_{w_t}\Re\big( e^{-\i \th}g_t^*\O\big)
&=
\iota_{w_t}\Re \big( e^{-\i \th}\det[I+\i\nabla^2_xk]\big)dx_1\w\cdots\w dx_n
\cr
&=\sum_{i=1}^n (-1)^i a^i(t,x)\Re \big(e^{-\i \th}\det B_0\big) dx_1\w\cdots \w\widehat{dx_i}\w\cdots\w dx_n,
\eaeq
$$
Comparing with equation~\eqref{eq:crfl}, we obtain
$$
a^i= - (-1)^i \frac{\Re \big(e^{-\i \th}\det B_i\big)}{\Re\big( e^{-\i \th}\det B_0\big)}.
$$
Thus, the geodesic equation becomes
$$
-\ddot k-\sum_{i=1}^n (-1)^i
\frac{\Re\big( e^{-\i \th}\det B_i\big)}{\Re\big( e^{-\i \th}\det B_0\big)} \del_{x_i} \dot k=0,
$$
or
$$
\ddot k\Re (e^{-\i \th}\det B_0)+\sum_{i=1}^n (-1)^i {\Re (e^{-\i \th}\det B_i)} \del_{x_i} \dot k
=
\Im e^{-\i \th}\det [I_n+\i\nabla^2 k]=
0.
$$
Here, in the last step, $I_n$ is the $(n+1)$-by-$(n+1)$ matrix
$\diag(0,1,\ldots,1)$ and we have replaced $\dot k$ by $\i\dot k$ and $\Re$
by $\Im$.

In summary, we have shown the following.

\bprop
\label{SPLgeodEq}
Let $\th\in(-\pi,\pi]$ and let $k_i\in C^2(\RR^n), i=0,1,$ be such that
$\graph(dk_i)\subset\C^n$ are elements
of $\calO_\th$. Let $k\in C^2([0,1]\times\RR^n)$ be such that
$\graph(d_xk(t,\,\cdot\,))\subset\C^n$ is an element of $\calO_\th$
for each $t\in[0,1]$.
Then $t\mapsto \graph(d_xk(t,\,\cdot\,))$
is a geodesic in $(\calO_\th,(\,\cdot\,,\,\cdot\,))$
with endpoints $\graph(dk_i), i=0,1$, if and only if $k$ satisfies
\begin{gather}
\label{DSLEq}
\Im \left( e^{-\i \th} \det (I_n + \i\nabla^2 k)\right) = 0,  \qquad \Re \left(e^{-\i \th} \det \left(I + \i\nabla_x^2 k\right)\right) > 0,
\cr
k(0,\,\cdot\,)=k_0+c,\qquad k(1,\,\cdot\,)=k_1+c,
\end{gather}
for a constant $c \in \R.$
\eprop

\section{The space-time Lagrangian angle}
\label{DDSec}

As shown in the previous section,
%\S\ref{SPLSec},
the degenerate special
Lagrangian equation (DSL)
governs $C^2$ geodesics in $(\calO_\th,(\,\cdot\,,\,\cdot\,))$.
Our goal in the next few sections
is to understand some of
the basic analytic properties of this equation.

Let $\tan^{-1}$ denote the branch of
the inverse to $\tan$ with image in $(-\pi/2,\pi/2),$
and let $\arg$ denote the branch of the argument function with image in $(-\pi,\pi].$
Then $\tan^{-1}\la:=\arg(1+\i\la)$,
for $\la\in \R$.
For a matrix $A\in\Symn$, denote by
$\la_1(A),\ldots,\la_n(A)$ its (real)
eigenvalues, with associated eigenvectors
$v_1(A),\ldots,v_n(A)$. Denote by
$\tan^{-1}A$ the matrix
whose eigenvalues are $\tan^{-1}\la_j(A)$
with associated eigenvectors $v_j(A),\, j=1,\ldots,n.$
The eigenvalues of $I+\i A$ are
$1+\i\la_j(A), j=1,\ldots,n$. This is because $I$ and $A$ are simultaneously diagonalizable. Therefore, all the eigenvalues of $I+\i A$ lie in a line in $\CC$ that is strictly contained in the right half space, and we can define $\arg(I + \i A)$ as we defined $\tan^{-1}(A)$ and $\arg(I + \i A) = \tan^{-1}(A).$ Moreover, $\tan^{-1}(A) = \arg(I+\i A)$ is a well-defined real-analytic matrix-valued function of $A$ \cite[p. 44]{Kato}. This observation is the basis for Harvey--Lawson's~\cite{HL} study of the special Lagrangian (SL) equation
in $\CC^n$.

The Lagrangian angle $\th_u:\RR^n\ra S^1$
of the Lagrangian
$\graph(\nabla u)$ for $u\in C^2(\RR^n)$ is given by
\begin{equation}\label{eq:LA}
\theta_u = \arg \det(I + \i\nabla^2 u).
\end{equation}
It can be lifted to the function $\tilde \th_u:\RR^n\ra \RR$
by the explicit formula
\beq
\label{liftedSLangleEq}
\tilde\th_u(x):=\tr\arg(I + \i \nabla^2u(x)) = \tr\tan^{-1}(\nabla^2u(x)).
\eeq
Solutions of the SL are thus
equivalent to solutions of $\tilde\th_u=c$
for $c=\th \, \mod\, 2\pi$.
Each of the possible choices of $c = \theta \mod 2\pi$ defines a {\it branch}
of the SL of angle $\theta.$
%Since the image of $\th_u$
%is contained in $(-n\pi/2,n\pi/2)$
Our goal in \S\ref{s-tLagSubSec} is to associate to each function $k \in C^2(\R^{n+1})$ an angle $\Th_k : \R^{n+1} \to S^1$ associated with the DSL equation, and construct a lift $\wtThk : \R^{n+1} \to \R.$ Unlike $\theta_u,\tilde\theta_u,$ the angle $\Th_k$ is not defined at critical points of $\dot k:=\del_t k$, and the lift $\wtThk$ is only upper semi-continuous. As in the case of SL, the lift $\wtThk$ gives rise to distinct branches of DSL. Appendix~\ref{LimitDSet} derives an alternative formula for the
lifted space-time Lagrangian angle $\wtTh_k$.

\subsection{The space-time Lagrangian angle and its lift}
\label{s-tLagSubSec}

Let $D\subset\RR^n$. Given $k\in C^2([0,1]\times D)$,
we define the {\it space-time Lagrangian angle}
of $k$ by
\[
\Th_k(t,x) = \arg\det(I_n + \nabla^2 k(t,x)) \in S^1,
\]
for $(t,x)$ such that $\det(I_n + \nabla^2 k(t,x))\neq 0.$ Lemma~\ref{RLemma} below shows that $\det(I_n + \nabla^2 k(t,x)) = 0$ if and only if $(t,x)$ is a critical point of $\dot k.$

In order to construct a subequation associated to DSL, we need to lift $\Th_k$ to a continuous real valued function and extend it to be upper semi-continuous over the critical points of $\dot k.$ To state the key technical result that constructs such a lift, we set the following notation.
For
$$
A=[a_{ij}]_{i,j=0}^n\in \Symnplusone,
$$
let
\beq
\label{AplusEq}
A^+:=[a_{ij}]_{i,j=1}^n\in \Symn,
\eeq
and
\beq
\label{avecEq}
\vec a_0:=(a_{01},\ldots,a_{0n})\in \RR^n.
\eeq
Write
\beq
\label{calSEq}
\calS:=\{A\in \Symnplusone\,:\, A=\diag(0,A^+)\}.
\eeq
For $B \in \Sym(\CC^m),$ denote by $\spec(B)$ the set of its eigenvalues, and for $\lambda \in \spec(B),$ denote by $m(\lambda)$ the multiplicity of $\lambda$ as a root of the characteristic polynomial of $B.$ Let
\begin{equation}\label{eq:whTh}
\widehat\Theta : \Symnplusone\sm\calS \to \RR, \qquad \qquad
\widehat \Theta(A) = \sum_{\lambda \in \spec(I_n + \i A)} m(\lambda)\arg(\lambda),
\end{equation}
and let $\wtTh,\underline\wtTh: \Symnplusone \to \RR$ be given by
\begin{gather*}
\wtTh(A) = \underline\wtTh(A) = \widehat\Theta(A), \qquad A \in \Symnplusone\sm\calS, \\
\wtTh(A) = \pi/2 + \tr\arg(I+\i A^+), \qquad \underline \wtTh(A) = -\pi/2 + \tr\arg(I + \i A^+), \qquad A\in\calS.
\end{gather*}
\bthm
\lb{USCThm}
The function $\widehat\Theta$ is well-defined and differentiable. Moreover, $\wtTh$ (resp. $\underline\wtTh$) is the smallest upper semi-continuous (resp. largest lower semi-continuous) function on $\Symnplusone$ extending $\widehat\Theta.$
\ethm

Consequently, we make the following definition.

\bdefin
Let $k\in C^2([0,1]\times D)$.
The \emph{regularized lift} of the space-time Lagrangian angle~$\Theta_k$
is the upper semi-continuous function
$$
\wtThk(t,x):=\wtTh(\nabla^2k(t,x))
\in (-(n+1)\pi/2,(n+1)\pi/2).
$$

\edefin

Subsections \ref{ArgSubSec}--\ref{USCSubSec} are devoted
to the proof of Theorem \ref{USCThm}.

\subsection{The argument of certain matrices}
\label{ArgSubSec}

We recall the following observation (cf. \cite[p. 94]{HL}
for the case $\delta=1$).

\blem
\lb{HLPosLemma}
Let $C\in\Sym(\RR^m)$ and $\delta>0$.
Then
$\Re\big((\delta I+\i C)^{-1}\big)$ is positive definite
\elem

Indeed, if $O\in O(m)$ diagonalizes $C$ so that
$C=O^T\diag(\la_1(C),\ldots,\la_m(C))O$,
then
\beq\label{RePartInvEq}
\Re\big((\delta I+\i C)^{-1}\big)
=
O^{T}\diag\bigg(
\frac\delta{\delta^2 +\lambda^2_1(C)},\ldots,
\frac\delta{\delta^2+\lambda^2_m(C)}\bigg)O.
\eeq

For the remainder of this subsection we let $A\in \Symnplusone$.
Also we will denote by $B$ the matrix
$$
B:=I_n+\i A\in \Sym(\C^{n+1}),
$$
and define $B^+$ and $\vec b_0$ in a manner similar
to Equations \eqref{AplusEq}--\eqref{avecEq}.
%We denote the eigenvalues of $B$ by
%$$
%\la_i, \q i=0,\ldots,n.
%$$

We would like to define the argument of matrices of the form
$I_n+\i A$. Clearly, $I_n$ and $A$ are not, in general, simultaneously
diagonalizable. Thus, the discussion at the beginning of Section~\ref{DDSec} does not
apply. To overcome this difficulty we start with the following observation. Write
\begin{equation}\label{eq:Ine}
I^\eta_n:=\diag(\eta,1,\ldots,1)\in\Sym(\R^{n+1}).
\end{equation}

\blem
\label{RealPartEVLemma}
Let $A\in \Symnplusone$ and $B=I_n^\eta +\i A\in \Sym(\C^{n+1})$ for $\eta \geq 0.$
Then the eigenvalues $\{\la_i\}_{i=0}^n$ of $B$ satisfy
$\Re\la_i\ge 0$. If $\vec a_0\neq0$ or $\eta > 0$,
then $\Re\la_i> 0$.
\elem

\bpf
Consider
$$
D=[d_{ij}]_{i,j}:=I_n^\eta+\i A+\delta I - \gamma \i I.
$$
The first claim follows if we can show that $D$
is nonsingular for each $\delta>0$ and $\gamma\in\RR$.

For $C=[c_{ij}]_{i,j=0}^n\in \Sym(\C^{n+1})$ with
$C^+$ invertible,
\beq
\label{DetFormulaEq}
\det C
=
\det C^
+
\big(
c_{00}-\langle \vec c_0, (C^+)^{-1}\vec c_0\rangle
\big).
\eeq
We apply this to $D$ as follows.
First, $D^+=(1+\delta)I+\i(A^{+}-\gamma I)$
is invertible since its eigenvalues
are $1+\delta + \i (\la_i(A^+)-\gamma)\neq0$,
as $\delta>0$. Thus $\det D\neq0$ iff
$
d_{00}-\langle \vec d_0, (D^+)^{-1}\vec d_0\rangle\neq0.
$
Now,
$$
\Re \left(d_{00}-\langle \vec d_0, (D^+)^{-1}\vec d_0\rangle\right) = \delta+\eta +
\langle \vec a_0, \Re\big((D^+)^{-1}\big)\vec a_0\rangle,
$$
which is positive by Lemma \ref{HLPosLemma}.
This proves the first statement.
Whenever $\vec a_0\neq0$ or $\eta >0$, positivity persists for $\delta=0$,
proving the second statement.
\epf

\bcor\label{cor:argB}
Suppose that $A\neq \diag(0,A^+)$. There exists a
closed simply-connected smooth
contour $\gamma$ entirely contained in
$\CC\sm\RR_{\leq 0}$
enclosing all the eigenvalues of $B=I_n+\i A$.
The function
$$
\arg(B):=\frac1{2\pi\i}
\int_\gamma (\zeta I-B)^{-1}\arg\zeta d\zeta
$$
is then well-defined independently of the choice of such
a contour $\gamma$. Moreover, it is a differentiable
function of $A\in\Symnplusone$ whenever $A\neq \diag(0,A^+)$.
\ecor
\bpf
Whenever $\vec a_0\neq0$,  Lemma \ref{RealPartEVLemma}
implies the existence of a contour $\gamma$ as in the statement
since then the eigenvalues of $B$ are contained in
$\RR_{>0}\times\i\RR
\subset
\RR_{\geq 0}\times\i\RR\sm\{(0,0)\}$.
If $\vec a_0=0$ but $A\neq \diag(0,A^+)$  then
$a_{00}\neq0$ and $A=\diag(a_{00},A^+)$.
Thus $\i a_{00}$ is an eigenvalue of $B$
and the remaining eigenvalues are in
$\{1\}\times\i\RR$. Thus, once again,
all the eigenvalues of $B$ are contained in
$\RR_{\geq 0}\times\i\RR\sm\{(0,0)\}$.
In conclusion, whenever $A\neq\diag(0,A^+)$, the
Dunford--Taylor integral in the statement is well-defined
independently of $\gamma$ since the branch of the argument function
with values in $(-\pi,\pi]$ is smooth away from the slit $\RR_{\leq 0}$.
Differentiability as a function of $A$ follows from  differentiation
under the integral sign.
\epf

On the other hand, by \cite[p. 45]{Kato} the eigenvalues of $\arg(B)$ are the arguments of the eigenvalues of $B$ with corresponding multiplicities. In particular,
\begin{equation*}\label{eq:trsumeig}
\sum_{\lambda \in \spec(I_n + \i A)} m(\lambda)\arg(\lambda) =  \tr \arg(I_n + \i A).
\end{equation*}
So, Corollary~\ref{cor:argB} proves the first part of
Theorem~\ref{USCThm}. That is, $\widehat\Theta$ is well-defined and differentiable.

\subsection{Upper semi-continuity of the lifted
space-time Lagrangian angle}
\label{USCSubSec}

The purpose of this subsection is to complete the proof
of Theorem \ref{USCThm}, namely,
to show that $\wtTh$ (resp. $\underline\wtTh$)
is the minimal usc (resp. maximal lsc) extension of
$\widehat\Theta$ from $\Symnplusone\sm\calS$ to $\Symnplusone$. We treat only $\wtTh$. The argument for $\underline \wtTh$ is analogous.

Indeed, minimality is
immediate since $\diag(\eps,A^+)\not\in\calS$ and for $\eps>0$,
$$
\tr\arg[I_n+\i\diag(\eps,A^+)]
=\pi/2+\tr\arg[I+\i A^+].
$$
We now turn to establishing the upper semi-continuity.
Let $A_i\ra A\in\calS$ with $A_i\not\in\calS$.
Then the eigenvalues
of $I_n+\i A_i$ converge to those of
$I_n+\i A=\diag(0,I+\i A^+)$, i.e.,
to $0,1+\i\la_1(A^+),\ldots,1+\i\la_n(A^+)$.
Therefore, according to \cite[p. 45]{Kato}, $n$ of  the eigenvalues
of $\arg(I_n+\i A_i)$ converge to
$\arg(1+\i\la_1(A^+)),\ldots,\arg(1+\i\la_n(A^+))$.
The remaining eigenvalue is $\arg \delta_i$,
with $\delta_i\in\CC$ in the right half space
by Lemma \ref{RealPartEVLemma},
and with $\delta_i$ converging to $0\in\CC$.
%Since $A_i$ can be chosen so that $\delta_i$ is
%any point in a neighborhood of $0$, i
It follows
that $\limsup\arg\delta_i\le\pi/2$ and that
%there exists $A_i$ for which equality is attained.
%Thus,
$$
\limsup_i\tr\arg(I_n+\i A_i)\le\pi/2+\tr\arg[I+\i A^+].
$$
%with equality attained for some $A_i\ra A$.
This concludes the proof of Theorem \ref{USCThm}.

\subsection{Bounds on the space-time angle and its lift}
It will be useful to compare the space-time Lagrangian angle with the usual Lagrangian angle of space-like slices, and similarly for their respective lifts. For this a pointwise analysis suffices, so we frame our discussion in terms of functions of a symmetric matrix $A.$ We apply these results by taking $A$ the Hessian of a function. Let
\[
\theta: \Symn \to S^1, \qquad \Theta:\Symnplusone\sm\calS \to S^1,
\]
be given by
\[
\theta(A) = \arg \det(I + \i A), \qquad \Theta(A) = \arg\det(I_n + \i A).
\]
We consider $S^1$ as an abelian group and use additive notation for the group law and the inverse.
\begin{lemma}\label{lm:Thth}
For all $A \in \Symnplusone\sm\calS,$ we have
\[
\Theta(A) - \theta(A^+) \in [-\pi/2,\pi/2] \subset S^1.
\]
\end{lemma}
\begin{proof}
Formula~\eqref{DetFormulaEq} implies
\begin{equation}\label{eq:Thth}
\Theta(A) - \theta(A^+) = \arg\left(\i a_{00} + \left\langle \vec a_0,(I + \i A^+)^{-1}\vec a_0\right\rangle\right).
\end{equation}
But
\[
\Re\left(\i a_{00} + \left\langle \vec a_0,(I + \i A^+)^{-1}\vec a_0\right\rangle\right) = \left\langle \vec a_0, \Re(I + \i A^+)^{-1}\vec a_0\right\rangle \geq 0
\]
by Lemma~\ref{HLPosLemma}. So the right-hand side of equation~\eqref{eq:Thth} must belong to $[-\pi/2,\pi/2].$
\end{proof}
Let $\tilde \theta : \Symn \to \R$ be given by
\begin{equation}\label{eq:thetat}
\tilde\theta(A) = \tr\arg(I + \i A) =  \tr\tan^{-1}(A).
\end{equation}
\begin{lemma}\label{lm:tThth}
For all $A \in \Symnplusone,$ we have
\begin{equation}\label{eq:tThth}
\left|\wtTh(A) - \tilde\theta(A^+)\right| \leq \pi/2, \qquad \left|\underline\wtTh(A) - \tilde\theta(A^+)\right| \leq \pi/2.
\end{equation}
\end{lemma}
\begin{proof}
When $A \in \calS$, the definition of $\underline\wtTh,\wtTh,$ gives
\begin{equation}\label{eq:caseS}
\wtTh(A) - \tilde\theta(A^+) = \pi/2, \qquad \tilde\theta(A^+) - \underline\wtTh(A)  = \pi/2,
\end{equation}
which implies the claim. We deduce the case $A \notin \calS$ as follows. Let $\{A_t\}_{t\in[0,1]} \subset \Symnplusone$ be a continuous path with $A_0 = A$ and $A_1 \in \calS$ and $A_t \notin\calS$ for $t < 1.$ Recall that
\[
\theta(A^+) = \tilde\theta(A^+) \mod 2\pi, \qquad \Theta(A) = \wtTh(A) \mod 2 \pi,
\]
and $\widehat\Theta$ is continuous by Theorem~\ref{USCThm}. So, Lemma~\ref{lm:Thth} implies that either~\eqref{eq:tThth} holds with $A = A_t$ for all $t \in [0,1),$ or $|\widehat \Theta(A_t) - \tilde\theta(A_t)|\geq \pi$ for all $t \in [0,1).$ But the latter case is impossible because
\[
\limsup_{t \to 1} \left|\widehat \Theta(A_t)- \tilde\theta(A_t)\right| \leq \max\left\{\wtTh(A_1) - \tilde\theta(A^+_1),\tilde\theta(A^+_1) - \underline\wtTh(A_1) \right\} = \pi/2
\]
by Theorem~\ref{USCThm} and equation~\eqref{eq:caseS}.
\end{proof}

\section{Degenerate ellipticity}
\label{DegEllipSec}

Harvey--Lawson show that SL is elliptic
in the sense that its linearization is an elliptic operator
\cite[Chap. 3, Theorem 2.13]{HL82}. Here we establish the following theorem.

\bthm\label{SymbolThm}
Let $k\in C^2([0,1]\times D)$
%satisfy $\h{\rm graph}(d_xk(t,\,\cdot\,))\in\calO_\th$ for each $t\in[0,1]$.
be a solution of DSL (\ref{DSLMainEq}).
Away from the critical points of $\dot k,$ the symbol of the linearization of
\beq
\label{DSLOpEq}
u\mapsto
\Im \left( e^{-\i \th} \det (I_n + \i\nabla^2 u)\right)
\eeq
at $k$ is nonnegative with exactly one zero eigenvalue, and its nullspace is spanned by $\nabla \dot k.$ At the critical points of $\dot k,$ the linearization
is nonnegative with exactly one non-zero eigenvalue.
\ethm

\def\cof{\hbox{\rm cof}\,}
The proof of Theorem~\ref{SymbolThm} is given at the end of this section based on the following discussion.
Denote by
$
\cof B
$
the cofactor matrix associated to $B$.
The linearization $L_k$ of \eqref{DSLOpEq} at $k$
is given by
\beq\label{SymbolEq}
\baeq
L_k\psi:
&=
\frac{d}{ds}\Big|_{s=0}
\Big\{
\Im \big(e^{-\i\th}\det [I_n+\i\nabla^2 (k+s\psi)]\big)
\Big\}
\cr
&=
%\det[I_n+\i \nabla^2 k]
\tr\Big(
\Re\big(e^{-\i\th}\cof(I_n+\i\nabla^2 k)\big)
\nabla^2\psi
\Big).
\eaeq
\eeq
It remains to understand $\Re\big(e^{-\i\th}\cof(I_n+\i\nabla^2 k)\big).$ More generally, we consider
\[
A~\in~\Symnplusone, \qquad B = I_n +\i A.
\]
As usual, we use notation~\eqref{AplusEq}-\eqref{avecEq}. We first prove two general lemmas.
\blem
\label{RLemma}
If $\vec a_0 \not = 0$, then $B$ is invertible.
\elem
\bpf
This is an immediate consequence of Lemma~\ref{RealPartEVLemma}.
\epf

\blem
\label{DegPosLemma}
Suppose that $\vec a_0\neq0$. Then
$\Re B^{-1}$ is positive semi-definite with exactly one zero eigenvalue. The
nullspace is spanned by the first column of $A$.
\elem

\bpf
Taking the imaginary part of the equation
$ B B^{-1} = I$ gives
\beq
\label{ImPartofEq}
A\Re B^{-1}+I_n\Im B^{-1}=0.
\eeq
In particular, $(a_{00},\vec a_0)\in\ker \Re B^{-1}$.

We claim that in fact
$\ker\Re B^{-1}=\RR(a_{00},\vec a_0)$.
To that end, suppose that $v\in \ker\Re B^{-1}$.
Then Equation \eqref{ImPartofEq} gives
$$
0=I_n\Im B^{-1} v=I_n B^{-1} v.
$$
Since $B^{-1}$ is invertible, the kernel of $I_n B^{-1}$
is one-dimensional. This proves the claim
since $(a_{00},\vec a_0)\neq0$ by assumption.

Finally, we prove $\Re B^{-1}$
is positive semi-definite. Recall notation~\eqref{eq:Ine}.
Note that the matrices
 $I_n^{1/p^2}+\i A$ limit, as $p$ tends to infinity,
to $B=I_n+\i A$, and similarly for
the corresponding inverse matrices.
Now, for each $p>0$
$$
\Re
\big(
(I_n^{1/p^2}+\i A)^{-1}
\big)
=
I_n^p
\Re
\big(
(I+\i I_n^pAI_n^p)^{-1}
\big)
I_n^p
$$
is positive definite according to
Lemma \ref{HLPosLemma}. The lemma follows by taking $p$ to infinity.

\epf

For the next lemmas, we assume the following matrix version of DSL~\eqref{DSLMainEq},
\beq
\label{MatrixDSLEq}
\Im\big(e^{-\i\th}\det B\big)=0,
\qq \Re\big(e^{-\i\th}\det B^+\big)>0,
\eeq
which allows us to complete our analysis of $\Re\big(e^{-\i\th}\cof B\big)$.

We separate the discussion into two cases. The first case, treated in the following lemma, applies when $(t,x)$ is
a critical point of $\dot k$, i.e., $\nabla\dot k(t,x)=0$.
\blem
\label{CritPtCaseLemma}
Suppose \eqref{MatrixDSLEq} holds and that
$A=\diag(0,A^+)$.
Then $\Re\big(e^{-\i\th}\cof B\big)$ is nonnegative
with exactly one positive eigenvalue.
\elem

\begin{proof} We have
$B=\diag(0,B^+)$. Thus,
$\cof B = \diag(\det B^+,0,\ldots,0)$.
The lemma follows from~\eqref{MatrixDSLEq}.
\end{proof}

The second case, treated in the next several lemmas, applies when $(t,x)$ is not
a critical point of $\dot k$, i.e.,
$\nabla\dot k(t,x)\not=0$.

\blem
\label{FirstrowcolumnLemma}
Suppose \eqref{MatrixDSLEq} holds and that
$A\not=\diag(0,A^+)$. Then
$\vec a_0\not=0$.
\elem

\bpf
Suppose to the contrary that $\vec a_0=0$. Then
$$
\baeq
\Im \big(e^{-\i\th}\det B\big)
&=
\Im \Big(
e^{-\i\th}
\big(
\i a_{00} \det B^+
\big)
\Big)
\cr
&=
a_{00}
\Re
\big(
e^{-\i\th}
\det B^+
\big).
\eaeq
$$
Since by~\eqref{MatrixDSLEq} the left hand side vanishes while
$\Re
\big(
e^{-\i\th}
\det B^+
\big)>0$, we conclude that $a_{00}=0$. Thus we obtain a contradiction to the hypothesis $A\not=\diag(0,A^+)$.
\epf

\blem
\label{signeddetLemma}
Suppose \eqref{MatrixDSLEq} holds and that
$\vec a_0\neq0$.
Then $\Re\big(e^{-\i\th}\det B\big)>0$.
\elem

\bpf
By formula~\eqref{DetFormulaEq} and Lemma~\ref{HLPosLemma},
the equation
$\Im\big(e^{-\i\th}\det B\big)=0$ becomes
\beq
\label{ImfulldetEq}
\baeq
0
&=
a_{00}\Re\big(e^{-\i\th}\det B^+\big)
+\langle \vec a_0, \Re\big((B^+)^{-1}\big)\vec a_0\rangle
\Im\big(e^{-\i\th}\det B^+\big)
\cr
&\qq+\langle \vec a_0, \Im\big((B^+)^{-1}\big)\vec a_0\rangle
\Re\big(e^{-\i\th}\det B^+\big),
\eaeq
\eeq
while
\beq
\label{RefulldetEq}
\baeq
\Re\big(e^{-\i\th}\det B\big)
&=
-a_{00}\Im\big(e^{-\i\th}\det B^+\big)
+\langle \vec a_0, \Re\big((B^+)^{-1}\big)\vec a_0\rangle
\Re\big(e^{-\i\th}\det B^+\big)
\cr
&\qq -\langle \vec a_0, \Im\big((B^+)^{-1}\big)\vec a_0\rangle
\Im\big(e^{-\i\th}\det B^+\big).
\eaeq
\eeq
Since $\Re\big(e^{-\i\th}\det B^+\big)>0$, we may
solve for $a_{00}$ in \eqref{ImfulldetEq}. Substituting
this expression into \eqref{RefulldetEq} then yields
$$
\Re\big(e^{-\i\th}\det B\big)
=
\frac{
\big|\det B^+\big|^2
}
{
\Re\big(e^{-\i\th}\det B^+\big)
}
\langle \vec a_0, \Re\big((B^+)^{-1}\big)\vec a_0\rangle
$$
which is positive by Lemma \ref{HLPosLemma}
and \eqref{MatrixDSLEq}.
\epf

\begin{lemma}\label{lm:notCritPtcase}
Suppose \eqref{MatrixDSLEq} holds and that
$\vec a_0\neq0$. Then
$\Re\big(e^{-\i\th}\cof(B)\big)$ is positive semi-definite with exactly one zero eigenvalue. The
nullspace is spanned by the first column of $A$.
\end{lemma}
\bpf
By~\eqref{MatrixDSLEq}, we have
\[
e^{-\i\th}\det B=\Re\big(e^{-\i\th}\det B\big).
\]
So,
\[
\Re\big(e^{-\i\th}\cof(B)\big) = \Re\big(e^{-\i\th}\det B \,B^{-1}\big) =  \Re\left(e^{-\i\th}\det B\right) \Re B^{-1}.
\]
The claim follows from Lemmas~\ref{DegPosLemma} and~\ref{signeddetLemma}.
\epf

\begin{proof}[Proof of Theorem~\ref{SymbolThm}]
The theorem follows from equation~\eqref{SymbolEq}, and Lemmas~\ref{CritPtCaseLemma},~\ref{FirstrowcolumnLemma} and~\ref{lm:notCritPtcase}.
\end{proof}

\section{The subequation}
\label{DSetSec}
\subsection{Construction}
In this section we associate a subequation to the DSL.

Denote by $\Sym(\R^m)$ the set of all symmetric
$m$-by-$m$ matrices, and by $\P$ the subset of nonnegative matrices.
Following \HL \cite{HL}, a proper nonempty closed subset $F$ of $\Sym(\R^m)$ is a {\it subequation} (or a {\it Dirichlet set}) if
\eqref{DirichletSetPropEq} holds.
Denote by $\Int S$ the interior of a set $S$, and
by $S^c$ its complement.
The {\it dual set} to $F$, denoted by $\widetilde F$, is %defined by
$$
\tF:=(-\Int F)^c.
$$
As befits a notion of duality, $\tF$ is again a subequation,
and $\wt{\wt F}=F$ \cite[p. 408]{HL}.

%and its approximants.
% and let
% \[
% \argz(A) =
% \begin{cases}
% \frac{\pi}{2}\sign(a_{00}), & A \neq \diag(0,A^+), \\
% \frac{\pi}{2}, & A = \diag(0,A^+).
% \end{cases}
% \]
% Also, define
% $$
% \argt(A):=\argz(A) + \sum_{i=1}^n \tan^{-1}(\mu_i(A)).
% $$

\smallskip
Recall the definition of the lifted
space-time Lagrangian angle $\wtTh$ from
Theorem \ref{USCThm}.
For $c\in\R$, define $\calF_c \subset \Sym(\R^{n+1})$ by
\beq
\label{calFcEq}
\calF_c
:=\big\{A \in \Sym(\R^{n+1})\,:\,
%\usc\tr\arg(I_n+\i A)
\wtTh(A)\geq c
\big\}.
\eeq
\HL introduced
the set
\begin{equation}\label{eq:Fc}
F_c
:=\big\{A \in \Sym(\R^{n+1})\,:\,
\tr\tan^{-1}(A) \ge c\big\},
\end{equation}
in conjunction with the special Lagrangian equation.
When $|c|<(n+1)\pi/2$, the set $F_c$ is non-empty. Thus, $F_c$ is a subequation
because adding a positive semi-definite
matrix to $A$ does not decrease its eigenvalues,
and $\tan^{-1}$ is a monotonically increasing function.
\HL also show that \cite[Proposition 10.4]{HL}
\beq\label{FcDualEq}
\widetilde F_c=F_{-c}.
\eeq
%Note that both $F_c$ and $\calF_c$ are empty when $|c|\ge (n+1)\pi/2$.
We introduce $\calF_c$ to study the DSL. Building
on our work in the preceding sections,
we prove the following.

\begin{theorem}
\lb{calFcThm}
If $|c|<(n+1)\pi/2$, then $\calF_c$ is a subequation.
Its
dual is $\tcalF_c=\calF_{-c}$.
\end{theorem}
The proof is given in the following series of lemmas.
\begin{lemma}\label{lm:calFcc}
$\calF_c$ is closed and non-empty.
\end{lemma}
\bpf
By Theorem \ref{USCThm}, $\calF_c$ is a superlevel set of
a usc function. Hence, $\calF_c$ is closed.
It is nonempty since
$\wtTh(pI)=\pi/2+n\tan^{-1}p$ tends to $(n+1)\pi/2$
as $p$ tends to infinity, while $c<(n+1)\pi/2$;
thus $pI\in\calF_c$ for all $p\gg 1$.
\epf

\blem
\label{AlmostDSetLemma}
Suppose that $A\in \calF_c$. Then
$A+P\in\calF_c$ for each $P\in \P$.
\elem
\bpf
Since $\calF_c$ is closed by Lemma~\ref{lm:calFcc}, it suffices to prove $A + P \in \calF_c$ for $P$ positive definite and such that $A+P \neq \diag(0,A^++P^+).$ Suppose first that $A\neq\diag(0,A^+)$. Let $\{P_t\}_{t \in [0,1]}$ be a smooth path of matrices with $P_0 = 0, \,P_1 = P,$ such that $\dot P_t$ is positive definite for all $t,$ and such that $A + P_t \neq \diag(0,A^++P_t^+)$ for all $t.$ Indeed, the path $P_t$ can be constructed by starting with the linear path $t \mapsto tP$ and making a $C^1$ small perturbation to avoid the set of matrices $M$ satisfying $M = \diag(0,M^+),$ which has codimension at least $2.$
Then Theorem \ref{USCThm} implies
that $\wtTh(A + P_t)$ is differentiable for all $t.$
Using $\wtTh(A)=\tr\arg(I_n+\i A)=\tr\,\Im\!
\det\log(I_n+\i A)$, we calculate
$$
\frac{d}{dt}
\wtTh(A+P_t)=
\tr\Big(
\Re\big((I_n+\i (A+P_t))^{-1}\big)\dot P_t
\Big).
$$
This is nonnegative thanks to Lemma \ref{DegPosLemma}. Integrating from $t=0$
to $t=1$ yields
\beq
\label{DirichletSetDesiredEq}
\wtTh(A+P)\ge \wtTh(A)\ge c.
\eeq
Now, suppose $A = \diag(0,A^+).$ Choose $\epsilon > 0$ such that $\hat P = P - \epsilon I$ is positive definite and set $\hat A = A + \epsilon I$. Then
\[
\wtTh(\hat A) = \frac{\pi}2 + \tr\arg(I + \i (A^++\epsilon I)) > \tr\arg(I + \i A^+) = \wtTh(A).
\]
On the other hand, $\hat A \neq \diag(0,\hat A^+),$ so by the case of the lemma already proved, we conclude
\[
\wtTh(A + P) = \wtTh(\hat A + \hat P) \geq \wtTh(\hat A).
\]
Combining the preceding two equations, we again obtain inequality~\eqref{DirichletSetDesiredEq} as desired.
\epf

\begin{lemma}\label{lm:-wtTh}
For $A \in \Symnplusone$ we have
\[
\wtTh(-A) = -\underline\wtTh(A).
\]
\end{lemma}
\bpf
If $A\neq\diag(0,A^+)$, then
\[
\widehat \Theta(-A)=-\widehat \Theta(A).
\]
Indeed,
$\det[I_n+\i A-\delta I]=0$
iff $\det[I_n-\i A-\bar \delta I]=0$, and the multiplicities
of the eigenvalues
 are the same.
Therefore, $\wtTh(-A) = \widehat\Theta(-A) = -\widehat\Theta(A) = - \underline \wtTh(A).$
On the other hand, if $A = \diag(0,A^+)$ then
\[
\wtTh(-A) = \pi/2 + \tr\arg(I - \i A^+) = \pi/2 - \tr\arg(I + \i A^+) = - \underline\wtTh(A).
\]
\epf
\begin{lemma}\label{lm:Fcd}
We have $\widetilde \calF_c = \calF_{-c}.$
\end{lemma}
\bpf
Recall the assertion of Theorem~\ref{USCThm} that $\wtTh$ (resp. $\underline\wtTh$) is the minimal usc extension (resp. maximal lsc extension) of $\widehat\Theta$. It follows that
\[
\Int \calF_c = \left\{ A \in \Symnplusone\,:\,\underline\wtTh(A) > c\right\}.
\]
By Lemma~\ref{lm:-wtTh}, we obtain
\[
-\Int \calF_c = \left\{ A \in \Symnplusone\,:\, - \wtTh(A) > c\right\}.
\]
Therefore, $\tcalF_c=\calF_{-c}$ as claimed.
\epf
\begin{proof}[Proof of Theorem~\ref{calFcThm}]
The Theorem follows from the Lemmas~\ref{lm:calFcc},~\ref{AlmostDSetLemma}, and~\ref{lm:Fcd}.
\end{proof}

\subsection{Reformulation of the DSL}

The following result relates our efforts in this section with the DSL equation. It reformulates
the Dirichlet problem for $C^2$ solutions
of DSL \eqref{DSLMainEq} in terms of the
subequations $\calF_c,F_{c-\pi/2},$ and their duals.

\bcor
\label{CtwoDSLCor}
Let $\th\in(-\pi,\pi]$, let
$D$ be a domain in $\RR^n$,
and let $k\in C^2([0,1]\times D)$.
Then $k$ is a solution of the DSL \eqref{DSLMainEq} if and only if for each $(t,x)\in[0,1]\times D$,
\begin{align}
\nabla^2k(t,x) &\in
\calF_c\cap(-\calF_{-c})=\del\calF_c, \label{eq:pase}\\
\nabla^2_xk(t,x)
&\in
\Int \big(F_{c-\pi/2}\cap (- F_{-c-\pi/2}) \big), \label{eq:nxse}
\end{align}
for a fixed $c\in(-(n+1)\pi/2,(n+1)\pi/2)$ satisfying
$c=\th + 2\pi k$ with $k \in \Z$.
\ecor

\begin{proof}
Let $\th_{k(t)}:D\ra S^1$ denote
the Lagrangian angle associated to the Lagrangian
graph of $u = k(t)$ by formula~\eqref{eq:LA}.
Observe that
the condition
$$
\Re\big(e^{-\i\th}\det[I+\i \nabla^2_x k(t,x)]\big)>0
$$
is equivalent to $\th_{k(t)}(x)-\th \in (-\pi/2,\pi/2) \subset S^1$, which is equivalent to
\begin{equation}\label{eq:thetabd}
-\pi/2 < \tr\tan^{-1}(\nabla_x^2k(t,x))-c < \pi/2
\end{equation}
for an appropriate choice of $c = \theta \mod 2\pi.$ The preceding inequality is equivalent to condition~\eqref{eq:nxse}. For future reference, we rewrite equation~\eqref{eq:thetabd} using notation~\eqref{eq:thetat} as
\begin{equation}\label{eq:thetabdm}
|\tilde\theta(\nabla_x^2k(t,x)) - c| < \pi/2.
\end{equation}

We divide the remainder of the proof into two cases. First, consider the case $\nabla^2k(t,x)\not=\diag(0,\nabla_x^2k(t,x))$.
By Lemma~\ref{lm:-wtTh} condition~\eqref{eq:pase} is equivalent to
\begin{equation}\label{eq:whThk=c}
\widehat\Theta(\nabla_x^2k(t,x))=c,
\end{equation}
which implies
\begin{equation}\label{eq:imdet=0}
\Im\big(e^{-\i\th}\det[I_n+\i \nabla^2k(t,x)]\big)=0.
\end{equation}
Conversely, equation~\eqref{eq:imdet=0} implies equation~\eqref{eq:whThk=c} modulo $2\pi.$ So, Lemma~\ref{lm:tThth} and inequality~\eqref{eq:thetabdm} imply equation~\eqref{eq:whThk=c} holds exactly.

Second, consider the case
$\nabla^2k(t,x)=\diag(0,\nabla_x^2k(t,x)).$
Then $\det[I_n+\i \nabla^2k(t,x)]=0$, so the DSL is satisfied.
It remains to check that the DSL implies
condition~\eqref{eq:pase}.
Indeed, we have already shown the DSL implies condition~\eqref{eq:nxse}. But
$\nabla_x^2k(t,x)\in \Int F_{c-\pi/2}$
implies that
$\wtThk(t,x)=\pi/2+\tr\tan^{-1}(\nabla_x^2k(t,x))>c$,
which implies $\nabla^2k(t,x)\in\calF_c$.
Similarly,
$-\nabla_x^2k(t,x)\in \Int F_{-c-\pi/2}$
implies that $\wtTh(-\nabla^2k(t,x)) = \pi/2 + \tr\tan^{-1}(-\nabla_x^2k(t,x))>-c$,
which implies $-\nabla^2k(t,x)\in\calF_{-c}$.
\end{proof}

Motivated by Corollary \ref{CtwoDSLCor}, in the next two sections,
we define weak solutions for the DSL in terms of subequations.

\section{Dirichlet duality theory }
\label{DDTheoreySec}

%\subs{The $F$-Dirichlet problem}
%\lb{FDPSubSec}

\HL \cite{HL} develop a systematic way to solve, in a viscosity sense, the Dirichlet problem for possibly degenerate elliptic equations
involving only the Hessian by reformulating the problem in terms of a subequation $F$.
In this section we recall their main result and definitions (see also \cite{Nirenberg} for an exposition).

\subsection{Subequations and their associated functions}
\lb{DirichletSetsSubS}

In the rest of the article, $F$ will always stand for a subequation. Let $X$ denote an open connected subset of $\R^n$. A function $u \in \USC(X)$ is {\it subaffine}, denoted $u\in\SA(X)$, if for all affine functions $a$ and $K\subset X$ compact, $u\le a$ on $\del K$ implies $u\le a $ on $K$. \HL~\cite[Prop. 2.3]{HL} prove that
\begin{equation}\label{eq:mp}
u \in \SA(X) \qquad \Rightarrow \qquad \sup_X u \leq \sup_{\partial X} u.
\end{equation}
A function $u\in \USC(X)$ is of {\it type $F$}, denoted $u\in F(X)$, if $u+v\in\SA(X)$
for all $v\in C^2(X)$ satisfying $\n^2 v(x)\in \tF$, for all $x\in X$. From now on, unless stated otherwise, we assume that $X$ is bounded.

The definition of a type $F$ function is natural in the sense that
$u\in F(X)\cap C^2(X)$ iff $u \in C^2(X)$ and $\nabla^2u(x)\in F$ for all $x\in X$. To see this,
it suffices to note that on the level
of matrices, $A\in F$ iff $A+\wt F\subset \wt \P$,
because, by duality, the latter is equivalent to  $\P\subset \wt{A+\wt F}
=F-A$, but $A+\P\subset F$ is equivalent to $A\in F$.
\HL prove~\cite[Theorem 6.5]{HL} that actually
\beq\label{SAThm}
F(X)+\tF(X)\subset \SA(X).
\eeq
(The original definition only implies
this inclusion if one of the sets on the left is intersected with
$C^2(X)$) .

Note that $\P(X)\cap C^2(X)$ consists of the $C^2$ convex functions,
and
\[
\SA(X)\cap C^2(X)=\{u\in C^2(X)\,:\, \nabla^2 u \h{\ is nowhere negative}\}.
\]
%In general,
\HL show~\cite[Theorem 4.5]{HL} that
$\P(X)$ consists of the convex functions, while $\wt \P(X)=\SA(X)$.

%Finally, a function $u \in C^2(X)$ is called {\it strictly type $F$} if $\hess u(x) \in
%\Int F$ for all $x \in X.$
%If, moreover, $u-\eps|x|^2/2\in F(X)$ for some $\eps>0$,
%we say $u$ is {\it strictly type $F$ with modulus $\eps$}.
%Given $C>0$ we also denote $u\in C^{-1} F(X)$ if $Cu\in F(X)$.

% strict

%\sp duality and strong $F$\sp

\subsection{Boundary convexity}
\label{BoundaryConvSubSec}

Let $\R_{>0}:=\{x\in\R\,:\, x>0\}$.
Let $F$ be a subequation.
Define the {\it ray set associated to $F$} by
$$
\vF:=\cl{\{A\in\Sym(\R^n)\,:\,\R_{>0} A\cap F\not=\emptyset\}}.
$$
Then $\vF$ is a subequation satisfying
$A\in \vF$ iff $tA\in\vF$ for all $t\ge0$ \cite[Proposition 5.11]{HL}.

Suppose $\del X$ is smooth. Recall that $\II$, the second fundamental form of $\del X$ with respect to
 the inward pointing unit normal $N$, is a map
$\II_x: T_x\del X\ra T_x\del X$ defined by $dN_x(V)=\II_x(V) \mod N_x$,
for any $V\in T_x\del X$.
A domain $X\subset \R^n$ with smooth boundary is called {\it strictly $\vF$ convex} or simply {\it strictly $\vF$} if
the second fundamental form of $\partial X$ with respect to the inward pointing
unit normal satisfies
\beq
\label{IIvFEq}
\II_x = B|_{T_x\del X} \quad \text{for some} \quad B \in \Int\vF.
\eeq

Boundary convexity can also be formulated in terms of defining functions.
A smooth function $\rho\in C^\infty(\overline{ X})$ is
called a {\it defining function (DF)} if $ X=\{\rho<0\}$ and $\nabla\rho$
is nowhere zero on $\del X$.
Then, according to \cite[Corollary 5.4]{HL},
$ X$ is strictly $\vF$ iff there exists a smooth
DF $\rho$ such that for each $x\in \del X$,
\[
%\label{HessrhovFEq}
\nt\rho(x)|_{T_x\del X} = B|_{T_x\del X}\quad \h{for some\ } B\in \Int\vF.
\]
%Such a DF is called a strictly $\vF$ DF.
One checks that if such a DF exists, then
any DF has this property \cite[Lemma 5.2]{HL}.

\subsection{The \texorpdfstring{$F$}{F}-Dirichlet problem}
\label{FDPSubSec}

Let $F$ be a subequation, let $X\subset \RR^n$ be a bounded domain,
and let $\vp\in C^0(\del X)$.
The {\it $F$-Dirichlet problem for $( X,\vp)$} is the problem of finding a function $u\in C^0(\overline{X})$ solving
\beq
\label{FDPEq}
u\in F( X), \quad -u\in\tF( X), \quad u|_{\del X}=\vp.
\eeq
A function $u$ on $\overline{X}$ is called a \emph{subsolution} if $u \in F(X)\cap \USC(\overline{X})$ and $u|_{\del X}\leq\vp$. It is called a \emph{supersolution} if $-u\in\tF( X)\cap \USC(\overline{X})$ and $u|_{\del X}\geq \vp.$

To see how the $F$-Dirichlet problem relates to the usual Dirichlet problem
in a particular example,
consider the Laplace equation
$\tr \nabla^2 u = 0.$ The set $F = \{A \in \Symn\,:\,\tr A \geq 0 \}$ is a subequation,
in fact equal to $\tF$, and for $u \in C^2(X)$ the notions of sub/supersolution associated to $F$ coincide with the usual notion of a sub/supersolution for the Laplace equation.

The main existence  result of \HL is as follows \cite[Theorem 6.2]{HL}.
\begin{theorem}
\label{HLThm}
Let $ X\subset \R^n$ be a bounded domain with $\del X$ smooth
and both strictly $\vF$ and $\vtF$ convex. Then the \FDP \eqref{FDPEq}
admits a unique solution in $C^0(\overline{X})$.
%,  continuous up to the boundary.
\end{theorem}
%Note that $u$ is a solution to the $F$-Dirichlet problem for $(\O,\vp)$ if only if $-u$
%is a solution to the $\tF$-Dirichlet problem for $(\O,-\vp)$.

\subsection{Properties of functions of type \texorpdfstring{$F$}{F}}
We recall several results used by \HL in the proof of Theorem~\ref{HLThm} that will be important in the arguments presented below. Let $X \subset \R^n$ be a bounded domain and let $F$ be a Dirichlet set. The upper semi-continuous regularization of a function $u$ will be denoted by
$$
\usc u(x):=\lim_{\delta\ra0}\sup_{y\in X\atop |y-x|<\delta}u(y).
$$

\HL show~\cite[(4), p.406]{HL} that if $B \in \Sym(\R^n)$, then there exists $t_0 \in \R$ such that $B + tI \in F$ iff $t \geq t_0.$ As a consequence we obtain two properties of functions on $X$ of type $F$:
\begin{enumerate}[label = (S\arabic*)]
\item\label{it:quad}
There exists a constant $C > 0$ depending only on $F$ such that the function $C|x|^2$ belongs to $F(X).$
\item\label{it:ubd}
There exists a constant $C > 0$ depending only on $F$ such that for all $u \in F(X),$ we have $u + C|x|^2 \in \SA(X).$
\end{enumerate}
Property~\ref{it:ubd} follows from property~\ref{it:quad} applied to $\tF$ and the definition of $F(X)$~\cite[Lemma~6.6]{HL}.
The following properties of functions on $X$ of type $F$ are due to \HL~\cite[p. 410]{HL}.
\begin{enumerate}[label = (S\arabic*),resume]
\item\label{it:aff}
If $u \in F(X)$ and $a$ is affine, then $u + a \in F(X).$
\item\label{it:dl}
If $u_j \in F(X)$ satisfy $u_j \geq u_{j+1},$ then $\lim_{j \to \infty} u_j \in F(X).$
\item\label{it:uniflimit}
If $u_j \in F(X)$ converge in $C^0$ on compact sets
to $u$, then $u \in F(X).$
\item\label{it:ue}
Suppose $\calE \subset F(X)$ is locally uniformly bounded above. Let $u$ be defined by
\[
u(x) = \sup_{f \in \calE} f(x).
\]
Then $\usc u \in F(X).$
\item\label{it:hess}
If $u\in F(X)$ is twice differentiable at $x \in X$, then $\nabla^2 u(x) \in F.$
\end{enumerate}
The uniqueness part of Theorem~\ref{HLThm} is actually a special case of the following result~\cite[Theorem 6.3]{HL}, which does not require any conditions on $\del X.$
\begin{theorem}\label{tm:un}
If $u,v,$ are two solutions of the $F$-Dirichlet problem for $(X,\vp),$ then $u=v.$
\end{theorem}
The proof of Theorem~\ref{tm:un} is immediate from inclusion~\eqref{SAThm} and the maximum principle~\eqref{eq:mp}.

We will also need the following theorem, due to \HL~\cite[Corollary 7.5]{HL}, which is a by-product of the proof inclusion~\eqref{SAThm}. A function $u$ on $X \subset \R^n$ is called \emph{$\lambda$-quasi-convex} if $v = u + \frac{1}{2}\lambda|x|^2$ is convex. A fundamental theorem of Alexandrov says that the second derivative of a quasi-convex function exists almost everywhere.
\begin{theorem}\label{tm:ae}
Suppose $u$ is locally quasi-convex on $X$ and $\nabla^2 u(x) \in F$ for almost every $x.$ Then $u \in F(X).$
\end{theorem}

\section{Dirichlet duality theory with weak boundary assumptions}
\label{DDweakSec}

Unfortunately, Theorem \ref{HLThm} does not apply in our setting
since the domain $\calD=(0,1)\times D$ (recall \eqref{calDEq})
is not strictly $\vec\calF_c$ convex.
The purpose of the present section is to generalize the work of \HL
described in \S\ref{DDSec}
to allow for weaker boundary assumptions.
This section can be read independently
of the rest of the article since it is applicable to arbitrary subequations.

%Sections \ref{DirichletSetsSubS}--\ref{FDPSubSec} recall the main relevant
%points of \HLno's
%Dirichlet duality theory \cite{HL}
In \S\ref{CvxMultSubSec} we extend the notion of strict convexity
to domains with corners, and construct corresponding
boundary defining functions.
In \S\ref{DirichletWeakBndrySubSec} we prove a result (Theorem \ref{tm:subs})
concerning the $F$-Dirichlet problem that
generalizes Theorem~\ref{HLThm} by replacing strict boundary convexity with assumptions on the boundary values.

\subsection{Boundary convexity for nonsmooth boundary}
%multiple components}
\label{CvxMultSubSec}
In the following, we use several definitions concerning manifolds with corners. We follow the conventions of Joyce's article~\cite{Jo09}, to which we refer the reader for further details. Recall that the boundary $\partial X$ of a manifold with corners $X$ is itself a manifold with corners, equipped with a map
\[
i_X : \partial X\to X,
\]
which may not be injective. For example, think of $X=[0,1]\times[0,1]$
for which $\del X$ consists of four copies of $[0,1]$, so
the inverse image of $(0,0)\in X$ consists of two points in $\del X$. We say that $X$ is a manifold with embedded corners if $\partial X$ can be written as the disjoint union of a finite number of open and closed subsets on each of which $i_X$ is injective. For example, a teardrop shape is a manifold with corners, but not a manifold with embedded corners.
A function $\vp$ on $\del X$ is called \emph{consistent} if it
is constant on fibers of $\iota_X$.
Given a function $u : X \to \R$, we define its restriction to $\del X$ by
$u|_{\partial X}:= u \circ i_X.$

Let $ X \subset \R^n$ be a domain such that $\overline  X$ is a compact manifold with embedded corners. We denote by $\del X$ the boundary of $\overline{X}$ considered as a manifold with corners. In particular, each component of $\partial X$ is an embedded submanifold with corners of $\R^n.$ Let $\partial  X_i$ denote a connected component of $\partial  X.$

\begin{df}
\label{GlobalDFDef}
The boundary component $\partial  X_i$ is called strictly $\vF$ convex if
%\eqref{IIvFEq}
$$
\II|_{T_x\del X_i} = B|_{T_x\del X_i}\quad \h{for some\ } B\in \Int\vF,
$$
holds at each $x\in \partial  X_i$.
\end{df}
A smooth function $\rho$ defined near a point $x\in \del X_i$ is
said to be a local defining function for $\del X_i$ near $x$ if
on some neighborhood $U$ of $x$ we have $X\cap U=\{\rho <0\}$
and $\nabla \rho\not=0$. By the proof of \cite[Lemma 5.3]{HL}, the boundary component $\del X_i$ is strictly $\vF$ convex iff
there exists a local defining function for $\del X_i$ satisfying
\[
%\label{HessrhovFEq}
\nt\rho(x)|_{T_x\del X_i} = B|_{T_x\del X_i}\quad \h{for some\ } B\in \Int\vF,
\]
near each $x\in \del X_i$.
One checks that if such a DF exists, then
any DF has this property \cite[Lemma 5.2]{HL}.

\begin{df}
%\label{GlobalDFDef}
A function $\rho \in C^\infty(\overline X)$ is called a {global defining function} for $\partial X_i$ if
\beq\lb{PartialDFEq}
\rho|_{\overline{X}\setminus \partial X_i} < 0,
 \qquad \rho|_{\partial  X_i} = 0,\qquad \n \rho|_{\partial  X_i} \neq 0.
\eeq
%Such a function is said to be strictly type $\vF$ whenever
%$$
%\nabla^2\rho(x)|_{T_x\del X_i} = B|_{T_x\del X_i}\quad \h{for some\ } B\in
%\Int\vF,
%$$
%for each $x\in \del X_i$.
\end{df}

Note that
\beq\label{rhonegativeEq}
\rho|_{\overline X} \le 0,
\eeq
but $\rho$ only vanishes on $\partial X_i.$

The next result shows that Definition \ref{GlobalDFDef} can be interpreted
in terms of strictly $\vF$ convex
global defining functions, just as in the setting of a smooth boundary described in
\S\ref{BoundaryConvSubSec}. Moreover, it shows that
an analogue of \cite[Theorem 5.12]{HL} concerning the existence
of uniformly $\vF$ convex defining functions holds in our setting.

\begin{prop}
\lb{StrictFPartialProp}
If the boundary component $\partial X_i$ is strictly $\vF$ convex, then there exists a global defining function $\rho \in C^\infty(\overline  X)$ for $\partial X_i$ that is stricty type $\vF$.
Moreover, there exists $\eps, R>0$ such that %for all $C\ge R$,
\beq\lb{StrictFPartialEq}
C\left(\rho - \epsilon |x|^2\right) \in F(\overline{X}) \text{ for all } C \geq R.
\eeq
\end{prop}

\bpf
We start by constructing a global boundary defining function for $\partial X_i.$
Let $\RR_+:=\{x\in\RR\,:\, x\ge 0\}$.
As a manifold with corners, $\overline X$ comes equipped with coordinate charts
the domain of each of which is an open set in $\RR_+^n$ \cite{Jo09}.
Let $U_j\subset \RR^n_+$
and let
\beq\label{psiiEq}
\psi_j:U_j\ra \overline X,
\eeq
be a
collection of charts
such that $i_X(\del X_i)\subset \cup \psi_j(U_j)$.
Let
$$
W_j:=\psi_j(U_j).
$$

We start by constructing smooth local DFs for $\del X_i$ defined on each $W_j$.
Since $\overline X$ is a manifold with \emph{embedded} corners, there exists $l\in\{1,\ldots,n\}$ such that
$$
\psi_j^{-1}(\del X_i\cap W_j)=
\{x \in U_j\,|\,
x_{l}=0\}.
$$
Define a smooth function on $W_j$ by (recall \eqref{psiiEq})
$$
f_j:= - x_l\circ \psi_j^{-1}.
$$
Note that $\nabla f_j(p) \not = 0$ for $p\in \del X_i\cap W_j$ and $f_j(p) < 0$ for $p \in W_j \cap (\overline X\setminus \del X_i).$

Let $U =\overline{X}\setminus \del X_i.$
Then $\{U, \{W_j'\}_j\}$ is a covering of $\overline X$ by open sets.
Consider a smooth partition of unity
$\alpha_U, \{\alpha_j\}_{j\in\N}$
subordinate to $\{U,\{W_j'\}_j\}$.
Then set
$$
\rho:=\sum\alpha_j f_j -\alpha_U\in C^\infty(\overline X).
$$
By construction,
(i) $\rho$ vanishes on $\del X_i$, (ii) $\rho<0$ on $\overline X \setminus \partial X_i$,
(iii) $\nabla\rho\not=0$ on $\del X_i$.

We now construct a
strictly $\vec F$ global boundary DF by following the
argument of
\cite[Theorem~5.12]{HL}.
First, the argument of \cite{HL} shows that $\tilde \rho  = \rho+C\rho^2$ is a
strictly $\vec F$ local DF on a neighborhood
 of $\del X_i$ for all $C\gg 1$.
Since $\tilde \rho$ is negative in a neighborhood of $\del X_i,$ using a partition of unity argument, we can modify $\tilde \rho$ so it is negative on all of $\overline{X}\setminus \partial X_i$ and still strictly $\vec F$ near $\del X_i.$ Thus, we may replace $\rho$ by a global DF for $\del X_i$, still denoted by $\rho$,
that is strictly $\vec F$ on a neighborhood $W$ of $\del X_i$.
Choose $r>0$ small enough so $\{\rho > -r\}\subset W.$ By compactness of $\overline X,$ choose $\delta>0$ small enough that
$\delta|x|^2-r< 0$ on $\overline X$.
Define
$$
\hat \rho:=\max\{\rho,\delta|x|^2-r\}.
$$
So, $\hat \rho = \rho$ in a neighborhood of $\partial X_i$ where $\rho$ is strictly $\vec F$ and $\hat \rho = \delta |x|^2 - r$ on $\overline X \setminus W.$

We now smooth $\hat \rho$ just as in the proof of~\cite[Theorem 5.12]{HL} to obtain a new $\rho.$
The same arguments as there then prove that
this new $\rho$ has the following properties: (i) it is a global DF for $\partial X_i,$
(ii) $\rho\in C^\infty(\overline {X})\cap \Int \vF(\overline {X})$,
(iii) $\rho$ satisfies \eqref{StrictFPartialEq}.
\epf

\begin{ex}
The reason for considering each boundary component separately is the following example. Take $F = \P$ so $F$-convexity is convexity in the usual sense. Let $f : \R \to \R$ be given by $f(t) = t(t-1)$. Consider $ X = [0,1]^2$. The function $g :  X \to \R$ given by $g(s,t) = f(s)f(t)$ has a saddle point at each corner of $[0,1]^2.$ But any function $\rho : \overline X \to \R$ with $\rho|_X < 0, \rho|_{\partial X} = 0$ and  $\n\rho|_{\Int \partial  X} \neq 0,$ must be approximately $g$ up to rescaling near the corners of $ X.$ Of course, the example generalizes.
\end{ex}

\subsection{The \texorpdfstring{$F$}{F}-Dirichlet problem with weak assumptions on the boundary}
\lb{DirichletWeakBndrySubSec}

A typical result in the theory of degenerate
real/complex \MA equations is
that existence of a convex/psh solution is implied by
existence of a convex/psh subsolution
to the Dirichlet problem that attains the boundary values (see, e.g.,
the discussion in \cite{Nirenberg} for some references).
Theorem \ref{tm:subs} below, based on Harvey--Lawson's theory,
can be considered as a result of this flavor
in the more general setting of subequations.
Thus, for example, Theorem~\ref{tm:subs}
furnishes solutions of the homogeneous
real/complex \MA equation  in all branches.

Let $ X\subset \R^n$ be a bounded domain such that $ \overline  X$ is a manifold with embedded corners.

\begin{df}
Let $\vp \in C^0(\partial X)$ be consistent.
Recall that a subsolution of the
$F$-Dirichlet problem for $( X,\vp)$ is a function $u \in F( X)\cap \USC(\overline{X})$
such that $u|_{\partial  X} \leq \vp$.
A subsolution $u$ for $( X,\vp)$ is {\it $\delta$-maximal at $p\in \partial  X$}
if $u(p)\ge \vp(p)-\delta$, and {\it maximal at $p$} if  $u(p)= \vp(p)$.
\end{df}

\begin{df}
\lb{FvpConvexDef}
We say $\partial  X$ is strictly $(F,\vp)$-convex if we can decompose $\partial X$ as
the disjoint union $A \cup B$ where $A$ and $B$ are unions of components and satisfy the following:
\begin{enumerate}
\item\label{it:A} For each $p \in A$ and $\delta>0$ there exists a $C^0(\overline{X})$ subsolution of the
\FDP for $( X,\vp)$ that is $\delta$-maximal at $p.$
\item\label{it:B}
$B$ is strictly $\vF$ convex.
\end{enumerate}
\end{df}

\begin{rem}
\lb{AutomaticConvexityRemark}
Suppose $F \subset \P$. Then $\tP \subset \tF$, thus $\tP\subset \vtF$.
It follows that any hypersurface
is strictly $\vtF$ convex. Indeed, whatever $\II_x$ may be,
for any $\eps>0$, $\diag(\eps,\II_x)\in\Int\tP$
by definition.
\end{rem}

The main result of this section is the following natural generalization of
Theorem \ref{HLThm} allowing $\overline{X}$ to be a manifold with embedded corners that is not necessarily strictly convex.

\begin{theorem}\label{tm:subs}
Let $F$ be a subequation in $\Sym(\R^n)$, and let $X$ be a bounded domain in $\R^n$
such that $\overline{X}$ is a manifold with embedded corners.
Let $\vp$ be a
consistent continuous function on $\del  X$.
Assume $\partial X$ is strictly $(F,\vp)$-convex and strictly $(\tF,-\vp)$-convex.
Then the \FDP for $( X,\vp)$ admits a unique solution in $C^0(\overline X)$.
\end{theorem}

Before proving Theorem~\ref{tm:subs}, we prove the following lemma, which builds on an idea in the proof of~\cite[Lemma 6.8]{HL}.

\begin{lemma}\label{lm:dss}
Let $\vp$ be a consistent continuous function on $\del X.$ Let $x_0$ be a point of a boundary component $\del X_i \subset \del X$ that is strictly $\vF$ convex. Then there exists a $C^0(\overline X)$ subsolution $w$ of the \FDP for $(X,\vp)$ that is $\delta$-maximal at $p.$
\end{lemma}
\begin{proof}
Choose a boundary defining function $\rho$ for $\del X_i$ and constants $\eps,R,$ as in Proposition~\ref{StrictFPartialProp}, so
\[
C(\rho-\eps|x|^2)\in
F(X)
\]
when $C\ge R$. Adding an affine function, also
$C(\rho-\eps|x-x_0|^2)\in
F(X)
$ by property~\ref{it:aff}.
Thus, using \eqref{rhonegativeEq},
given $\delta>0$, there exists $C\ge R$ sufficiently large such that
for any $x\in\del X$,
\beq\label{vpCInEq}
-\vp(x)+C(\rho(x)-\eps|x-x_0|^2) \leq -\vp(x) -C\eps|x-x_0|^2 \le -\vp(x_0)+\delta.
\eeq
Take
\[
w = C(\rho(x)-\eps|x-x_0|^2) + \vp(x_0) - \delta.
\]
Then inequality \eqref{vpCInEq} implies that $w|_{\partial X} \leq \vp$, and the vanishing of $\rho$ on $\del X_i$ implies that $w$ is $\delta$-maximal at $x_0.$
\end{proof}

\begin{proof}[Proof of Theorem~\ref{tm:subs}]
The set
$$
\calE_\vp:=\{v \h{ a subsolution to \FDP for $( X,\vp)$}\}
$$
is non-empty: For sufficiently large $C_1>0$ the function $C_1|x|^2$ belongs to $F(X)$ by property~\ref{it:quad}. Thus $C_1|x|^2-C_2\in\calE_\vp$ for
a sufficiently large constant $C_2$ depending only on $C_1$ and $||\vp||_{C^0}$.
Moreover, by property~\ref{it:ubd} there exists a uniform constant $C>0$ depending only on the Dirichlet data
and $X$
such that $u\le C$ for any $u\in\calE_\vp$.
For each $x\in\overline{X}$, set
$$
u_\vp(x):=\sup\{v(x)\,:\, v\in\calE_\vp\}.
$$
By property~\ref{it:ue}, we have $\usc u_\vp\in F( X)$.

{\it Step 1.}
We claim that $\usc u_\vp\in\calE_\vp$.
This implies that $\usc u_\vp \leq u_\vp$ by definition of $u_\vp,$ and thus $\usc u_\vp = u_\vp$ by definition of upper-semi-continuous regularization. Consequently, $u_\vp\in\calE_\vp$.

To prove $\usc u_\vp\in\calE_\vp$, let $x_0 \in \partial X$ and choose $\delta > 0.$ Depending on which case of Definition \ref{FvpConvexDef} we are in, either by assumption or by Lemma~\ref{lm:dss},
there
exists a $C^0(\overline{X})$ subsolution $w$ of the $\tF$-Dirichlet problem for $( X,-\vp)$
that is $\delta$-maximal at $x_0$. For any $v\in\calE_\vp$, we have $v+w\le 0$ on $\del X$.
Inclusion~\eqref{SAThm} implies that $v+w\in\SA(X)$. The maximum principle~\eqref{eq:mp} then gives $v\le -w$. So $u_\vp\le -w$, and since $w$ is continuous,
also $\usc u_\vp\le -w$. In particular, $\usc u_\vp(x_0)\le -w(x_0)\le \vp(x_0)+\delta$,
proving the claim, since $\delta>0$ is arbitrary.

{\it Step 2.} For every $x_0\in\del X$, $\liminf_{x\ra x_0}u_\vp(x)\ge \vp(x_0)$.

Choose $\delta > 0.$ Either by assumption or by Lemma~\ref{lm:dss}, there exists
a $C^0(\overline{X})$ subsolution $w$ of the \FDP problem for $(X,\vp)$
that is $\delta$-maximal at $x_0$. Since $u_\vp$ is the supremum of all subsolutions, we have $u_\vp \ge w.$ So, by the continuity of $w,$ we have
\[
\liminf_{x\ra x_0}u_\vp(x) \geq w(x_0) \geq  \vp(x_0)-\delta,
\]
and the claim follows since $\delta>0$ is arbitrary.

{\it Step 3.} $u_\vp\in C^0(\overline{X})$.

The proof of this property in \cite[Proposition 6.11]{HL} does not make any assumption on $\del X$
beyond it being a compact set, and so carries over to our setting.

{\it Step 4.} $u_\vp\in -\tF(X)$.

The proof of this is identical to that of \cite[Lemma 6.12]{HL}.

{\it Step 5.} Uniqueness.

This is a special case of Theorem~\ref{tm:un}.
%
%\sp this is identical to HL, so maybe omit this\sp
%Suppose $u_1,u_2$ are $C^0(\oO)$-solutions
%to the \FDP for $(\O,\vp)$. Then $u_1\in F(\O)$, while $-u_2\in \tF(\O)$.
%By \eqref{SAThm} then $u_1-u_2\in\SA(\O)$ and vanishes on $\delO$. The maximum principle
%implies that $u_1\le u_2$, and reversing the roles implies $u_1=u_2$.
\end{proof}

\section{Solution of the Dirichlet problem for the DSL}
\label{MainThmSec}

Finally, we are in a position to prove the existence and uniqueness of continuous solutions
to all branches of the Dirichlet problem for the DSL.

Let $\D \subset \R^n$ be a bounded domain with $\partial \D$ smooth.
Consider $\calD:= (0,1) \times \D,$ so $\overline \calD$ is a manifold with embedded corners.

\begin{theorem}\label{tm:cgeq}
Suppose $\partial D$ is strictly $\vF_{c-\pi/2},
\vtF_{c+\pi/2},$ convex.
Let $\vp \in C^{0}(\partial\calD)$ be consistent and
affine in $t$ when restricted to $[0,1]\times\del D$.
%J
% and Lipschitz at the boundary of the component $[0,1]\times\del D \subset \partial\calD$.
Consider the following hypotheses:
\begin{enumerate}
\item\label{it:one}
$c > -\frac{\pi}{2}$ and for each $i\in\{0,1\}$,
\beq\lb{BndryAssump1Eq}
\vp_i:= \vp|_{\{i\}\times D} \in C^{0}(D) \cap F_{c-\pi/2}(D).
\eeq
\item\label{it:two}
For each $i\in\{0,1\}$,
\beq\lb{BndryAssump2Eq}
\vp_i \in C^0(D)\cap F_{c-\pi/2}(D) \cap (-F_{-c-\pi/2}(D)).
\eeq
\end{enumerate}
If either~\ref{it:one} or~\ref{it:two} holds, there exists a unique solution in $C^0(\overline{\calD})$
to the $\calF_c$-Dirichlet problem for $(\calD,\vp)$.
\end{theorem}
This section is dedicated to the proof of Theorem \ref{tm:cgeq}.

\begin{remark}\label{rm:convexity}
For example, the boundary convexity assumptions hold for any $D$ with strongly convex smooth boundary.  For any $c,$ a solution of the $\calF_c$-Dirichlet problem is Lipschitz in $t$ by Lemma~\ref{tm:Lip}. If $c \in [n\pi/2,(n+1)\pi/2)$ (resp. $c\in (-(n+1)\pi/2,-n\pi/2]$), then a solution of the $\calF_c$-Dirichlet problem is convex (resp. concave) in $x$ by Lemmas~\ref{lm:rest} and~\ref{lm:contP} below. It follows that such a solution is Lipschitz in $x$ and $t.$
\end{remark}

First, we construct a subsolution to the DSL that is maximal
along certain components of the boundary. For $t_0 < t_1 \in \R,$ write $\calD_{t_0,t_1} = (t_0,t_1)\times D.$ Given $\vp_i \in C^0(D),$ define $v_i\in C^0(\overline \calD_{t_0,t_1}), \ i = 0,1,$ by
\beq\lb{viEq}
v_0(t,x) = \vp_0(x) - C(t-t_0), \qquad v_1(t,x) = \vp_1(x) - C(t_1-t).
\eeq

\begin{lemma}\label{lm:subs}
Suppose $\vp_i \in C^0(D) \cap F_{c-\pi/2}(D).$ For each $i\in\{0,1\}$,
the function $v_i$ defined in~\eqref{viEq} is of type $\calF_c.$
\end{lemma}
\begin{proof}
First, suppose
$\vp_i \in C^2(D).$
Then, for each $(t,x) \in \calD$,
$$
\wtTh(\nabla^2v_i(t,x))=\pi/2+\tr\tan^{-1}(\nabla_x^2\vp_i(x))
\ge c.
$$

\def\dist{\hbox{dist}}

Next, we treat the general case.
Let
$$
\vp_i^\eps(x):=
\sup_{y\in D}
\big[
\vp_i(y)-\eps^{-1}|x-y|^2
\big],
$$
and define $v_i^\eps$ by replacing
$\vp_i$ by $\vp_{i}^\eps$
%:=\vp_{\{i\}\times D}^\eps$
in
the definition \eqref{viEq} of $v_i$.
Let $D_\delta:=\{x\in D\,:\,
\dist(x,\del D)>\delta\}$.
Then, we have the following \cite[Theorem 8.2]{HL}.

\begin{enumerate}
\item
$\vp_i^\eps\in F_{c - \pi/2}(D_\delta)$ for $\delta(\eps)=C\sqrt{\eps}$
for $C$ depending only on $||\vp_i||_{C^0(D)}$.
\item\label{it:qc}
$\vp_i^\eps$ is $\frac1\eps$-quasiconvex.
\item\label{it:d}
$\vp_i^\eps$ decreases to $\vp_i$ as $\eps\ra0$.
\end{enumerate}
Properties \ref{it:qc} and \ref{it:d} carry over to $v_i^\eps$.
Quasiconvexity implies that the Hessian of $\vp_i^\eps$ exists a.e., so
$\nabla^2\vp_i^\eps(x)\in F_{c - \pi/2}$ for a.e. $x\in D_\delta$ by property~\ref{it:hess}.
Thus, the computation of the previous paragraph shows that
$$
\nabla^2v_i^\eps(t,x)\in
\calF_c \q\h{ for a.e. $(t,x)\in (t_0,t_1)\times D_\delta$}.
$$
Thus, $v_i^\eps\in\calF_c((t_0,t_1)\times D_\delta)$
for every $\eps>0$ by Theorem~\ref{tm:ae}. This implies
$v_i\in \calF_c(\calD_{t_0,t_1})$.
Indeed,
by definition (recall \S\ref{DirichletSetsSubS}),
we must check that $v_i+f\in\SA(\calD_{t_0,t_1})$ for any
$f\in C^2(\calD_{t_0,t_1})\cap \tcalF_c(\calD_{t_0,t_1})$.
But, $v_i^\eps+f\in\SA((t_0,t_1)\times D_\delta)$, and
since $v_i^\eps$ decreases to $v$
and $\lim_{\eps\ra0}\delta(\eps)=0$,
we have $v_i+f\in\SA(\calD_{t_0,t_1})$ by property~\ref{it:dl} as desired.
\end{proof}

\blemma\label{lm:tconv}
Let $\vp \in C^{0}(\partial\calD)$ be consistent and
affine in $t$ when restricted to $[0,1]\times\del D$.
Let $\delta>0$ and let $(t_0,x_0)\in [0,1]\times \del D$.
If $\del D$ is $\vF_{c-\pi/2}$ strictly convex, then there exists a subsolution to the
 $\calF_c$ Dirichlet problem for $(X,\vp)$ that is $\delta$-maximal at $(t_0,x_0)$.
\elemma
\bpf
Choose a boundary defining function $\rho$ for $\del D$ and constants $\eps,R,$ as in Proposition~\ref{StrictFPartialProp}, so
\begin{equation}\label{eq:r-}
C(\rho-\eps|x|^2)\in
F_{c-\pi/2}(D)
\end{equation}
when $C\ge R$. Adding an affine function, also
$C(\rho-\eps|x-x_0|^2)\in
F_{c-\pi/2}(D)
$ by property~\ref{it:aff}. Write $\vp_i = \vp|_{\{i\}\times D}$ for $i = 0,1.$
Using~\eqref{rhonegativeEq},
for any $\delta>0$, there exists $C\ge R$ sufficiently large that
for all $x\in D$ and $i = 0,1,$
\[
-\vp_i(x)
+C(\rho(x)-\eps|x-x_0|^2) \le -\vp_i(x_0)+\delta.
\]
Since by assumption $\vp|_{\del\calD}(t,x)=
(1-t)\vp_0(x)+t\vp_1(x)$, it follows that for each
$(t,x)\in\del\calD$,
\[
-\vp(t,x)
+C(\rho(x)-\eps|x-x_0|^2) \le -\vp(t,x_0)+\delta.
\]
So, $w(t,x):=C(\rho(x)-\eps|x-x_0|^2)+\vp(t,x_0)-\delta$
lies below $\vp(t,x)$
on $\del\calD$, and satisfies $w(t,x_0)= \vp_0(t,x_0)-\delta$
for each $t\in[0,1]$.
Since
\[
\nabla^2 w = \diag(0,\nabla_x^2 C(\rho-\epsilon|x-x_0|^2)),
\]
by condition~\eqref{eq:r-} we have
\[
\wtTh(\nabla^2w)=\pi/2+\tr\tan^{-1}(\nabla_x^2C(\rho-\epsilon|x-x_0|^2))
\ge c.
\]
Thus, $w\in \calF_c(\calD)$.
%
%Since $\del D$ is $\vF_c$ strictly convex there exists a smooth DF
%$\rho$ for $D$ such that $\nabla^2\rho(x)
%|_{T_x\del D} = B|_{T_x\del D}\quad \h{for some\ } B\in \Int\vF_c.$
%In particular, $\rho$ (now considered as a function of $(t,x)\in(0,1)\times \del D$
%that is independent of $t$) satisfies
%$$\nabla^2\rho(t,x)
%|_{T_{(t,x)}(0,1)\times \del D} = \diag(0,B)|_{T_{(t,x)}(0,1)\times \del D}.
%$$
%But $\diag(0,B)\in\Int \vcalF_c$ according to computations as in the previous page.
%Therefore, $(0,1)\times \del D$ is $\vcalF_c$ strictly convex.
\epf

\blemma
\lb{ConvAssumpLemma}
Let $\calD$ and $\vp$ be as in Theorem~\ref{tm:cgeq}. Then $\del\calD$ is $(\calF_c,\vp)$ strictly convex and $(\tcalF_c,-\vp)$ strictly convex.
\elemma

\bpf
We consider boundary components and the conditions they satisfy one by one. We take $\vp_i = \vp|_{\{i\}\times D}$ for $i = 0,1.$
\begin{enumerate}[label=(\alph*)]
\item\label{it:BB}
Since $\partial D$ is $\vF_{c-\pi/2}$ strictly convex and $\vp$ is affine on $[0,1]\times \partial D,$ Lemma~\ref{lm:tconv} implies that for each
$\delta>0$ and each
$(t_0,x_0)\in[0,1] \times \del D$
there exists a
subsolution to the $\calF_c$ Dirichlet problem for $(\calD,\vp)$ $\delta$-maximal
at $(t_0,x_0)$.
Similarly, since
$\partial D$ is $\vtF_{c+\pi/2} = \vF_{-c-\pi/2}$ strictly convex,
there exists a
subsolution to the $\calF_{-c}$ Dirichlet problem for $(\calD,-\vp)$
that is $\delta$-maximal
at $(t_0,x_0)$ for each
$\delta>0$.

%
%is both $\vcalF_c$ and $\vcalF_{-c} = \vtcalF_c$ strictly convex.
\item\label{it:AA+}
Take $t_0=0, t_1 = 1,$ in~\eqref{viEq}, and choose the constant $C$ large enough that $v_i |_{\partial \calD} \leq \vp.$ This is possible because $\vp$ is affine on $[0,1]\times\del D$. Then Lemma~\ref{lm:subs} shows that $v_i$ is a subsolution to the $\calF_c$ Dirichlet problem for $(\calD,\vp)$ maximal along $\{i\} \times D$ for $i = 0,1.$
\item\label{it:AA-2}
Under hypothesis~\ref{it:two} of Theorem~\ref{tm:cgeq}, we know $-\vp_i \in C^0(D) \cap F_{-c-\pi/2}.$ So, we apply Lemma~\ref{lm:subs} to $-\vp_i$. Again using the assumption that $\vp|_{[0,1]\times \del D}$ is affine and choosing the constant $C$ of~\eqref{viEq} large enough, we see that $v_i$ is a subsolution to the $\calF_{-c} = \tcalF_{c}$ Dirichlet problem for $(\calD,-\vp)$ that is maximal along $\{i\}\times D$ for $i = 0,1.$
\item\label{it:AA-1}
Under hypothesis~\ref{it:one} of Theorem~\ref{tm:cgeq}, the boundary components $\{i\}\times D$ are $\vtcalF_c$ strictly convex. Indeed, this amounts to finding a strictly $\vcalF_{-c}$ defining function for each of these components
of the boundary. Consider the function $f(t,x) = t(t-1)/2$, which is a DF for both components simultaneously.
Then $\wtTh(\nabla^2f(t,x)) = \pi/2,$ and since $-c < \pi/{2}$, we conclude that $f$ is strictly $\vcalF_{-c}$ as desired.
\end{enumerate}
Under hypothesis~\ref{it:one} of Theorem~\ref{tm:cgeq}, the Lemma follows from~\ref{it:BB},~\ref{it:AA+} and~\ref{it:AA-1}. Under hypothesis~\ref{it:two} of Theorem~\ref{tm:cgeq} it follows from~\ref{it:BB},~\ref{it:AA+} and~\ref{it:AA-2}.
\epf

\begin{proof}[Proof of Theorem~\ref{tm:cgeq}]
Combine Lemma \ref{ConvAssumpLemma} and Theorem~\ref{tm:subs}.
\end{proof}

\section{Regularity properties of solutions}
\label{RegSec}

In this section we prove that solutions to the DSL have some additional
regularity properties beyond continuity up to the boundary.
Corollary \ref{calFcConvCor} shows that the solution is
$F_{c-\pi/2}\cap (-F_{-c-\pi/2})$
convex
on each time slice. In Lemma \ref{tm:Lip} we prove the solution is Lipschitz
continuous in time together with an a priori estimate.

\subsection{\texorpdfstring{$F_{c-\pi/2}$}{F}-convexity along time slices}
\label{ConvFibSubSec}

Recall that $\calD = [0,1]\times D.$ Write $u_t = u|_{\{t\} \times D}.$
\begin{lemma}\label{lm:rest}
Suppose $u \in \calF_c(\calD)$. Then
$u_{t_0} \in F_{c-\pi/2}(D)$ for each $t_0 \in (0,1)$.
\end{lemma}

\begin{proof}
By definition (recall \S\ref{DirichletSetsSubS}) it suffices to show that
$v + u_{t_0}$ is subaffine for any
$v\in C^2(D) \cap \tF_{c-\pi/2}(D) = C^2(D) \cap F_{\pi/2-c}(D)$
(recall \eqref{FcDualEq}).

Fix $t_0 \in (0,1)$ and a compact set $K \subset D$.
Let $K_\delta:= [t_0-\delta,t_0+\delta] \times K\subset\calD$.
Let $a$ be an affine function on $D$ and
let $v \in C^2(D)  \cap F_{\pi/2-c}(D)$.
Let $\pi_2:(0,1)\times D\ra D$ denote the natural projection.
Given $\epsilon > 0,$ choose $\delta > 0$ small enough so that
\begin{equation}\label{eq:ed}
\max_{[t_0-\delta,t_0+\delta]\times\partial K}
(v\circ\pi_2 + u+ a\circ\pi_2) \leq \max_{\partial K} (v + u_{t_0} + a) + \epsilon.
\end{equation}
Let $C>0$ and define $v_C: \calD \to \R$ by
\[
v_C = v\circ\pi_2-C{(t-t_0)^2}.
\]
Then
\[
\wtTh(\nabla^2v_C(t,x))=-\pi/2+\tr\tan^{-1}(\nabla_x^2 v(x))
\ge -c,
\]
so $v_C \in \calF_{-c}(\calD).$ By Lemma~\ref{lm:Fcd}, it follows that $v_C \in \tcalF_c(\calD).$ Choose $C$ large enough so that
\begin{equation}\label{eq:Kd}
\max_{\partial K_\delta} (v_C + u+a\circ\pi_2) \leq \max_{\partial K} (v + u_{t_0}+a)+\epsilon.
\end{equation}
This is indeed possible: Inequality~\eqref{eq:ed} together with the fact
that $C>0$ takes care of the subset $[t_0-\delta,t_0+\delta]\times\partial K\subset \del K_\delta$, while choosing
$C$ large enough takes care of $\{t_0\pm \delta\}\times K \subset \del K_\delta$.
Since $u \in \calF_c(\calD)$ and $v_C\in \tcalF_{c}(\calD)$, we have $v_C + u \in \SA(\calD).$ Therefore,
using the fact that
\[
(v_C + u + a \circ \pi_2)|_{t = t_0} = v + u_{t_0} + a,
\]
as well as inequalities~\eqref{eq:mp} and~\eqref{eq:Kd}, we obtain
\begin{align*}
\max_K(v + u_{t_0} + a) &\leq \max_{K_\delta} (v_C + u + a\circ\pi_2) \\
&\leq \max_{\partial K_\delta} (v_C + u + a\circ\pi_2) \\
&\leq \max_{\partial K} (v + u_{t_0}+a)+\epsilon.
\end{align*}
Since $\epsilon$ was arbitrary, it follows that $v + u_{t_0}$ is subaffine as desired.
\end{proof}

The preceeding lemma implies that our viscosity solutions of the DSL actually preserve the relevant notion of convexity on slices $\{t_0\} \times D \subset \calD.$

\bcor\label{calFcConvCor}
Suppose $u$ is
the solution
of the $\calF_c$-Dirichlet problem for $(\calD,\vp)$
provided by Theorem~\ref{tm:cgeq}.
Then
$u_{t_0} \in
F_{c-\pi/2}(D)
\cap
(-F_{-c-\pi/2}(D))$ for each $t_0 \in (0,1)$.
\ecor

\begin{rem}
Neither Lemma~\ref{lm:rest} nor Corollary \ref{calFcConvCor} extend to
$t_0 \in \{0,1\}$.
For example, in the setting of Theorem~\ref{tm:cgeq}(i),
choose $\vp_i\in C^0(D)\cap F_{c-\pi/2}$ such that
$\vp_i\not\in -F_{-c-\pi/2}$.
Theorem~\ref{tm:cgeq}(i) then furnishes a function $u\in\calF_c(\calD)
\cap (-\calF_{-c}(\calD))$ satisfying $u|_{\{i\}\times D} =\vp_i$.
Thus, while $-u$ belongs to $\calF_{-c}(\calD)$,
the restriction $-u|_{\{i\}\times D}=-\vp_i$ does not belong to
$F_{-c-\pi/2}$.
This stems from the fact that $-u$ belongs to $\calF_{-c}$
only in the interior of $\calD$. Indeed, the
proof of \cite[Lemma 6.12]{HL} only applies to the interior of $\calD$.
\end{rem}

\subsection{A priori \texorpdfstring{$C^0$}{C0} estimate on time slices}
\begin{lemma}\label{lm:C0ts}
Suppose $u$ solves the $\calF_c$-Dirichlet problem for $(\calD,\vp)$ and
\[
\vp|_{[0,1] \times\partial D} \in C^{0,1}, \qquad \vp|_{\{i\}\times D} \in
F_{c-\pi/2}\cap \left(-F_{-c-\pi/2}\right), \quad i = 0,1.
\]
There exists a constant $C = C(\|\vp_1-\vp_0\|_{C^0(D)},\|\vp\|_{C^{0,1}([0,1]\times\partial D)})$ such that
\[
|\vp_1(x) - u(t,x)| \leq C(1-t), \qquad |u(t,x) - \vp_0(x)| \leq Ct, \qquad (t,x) \in \calD.
\]
\end{lemma}
\begin{proof}
Take $t_0=0, t_1=1,$ in equation~\eqref{viEq}, and choose $C$ large enough that $v_i|_{\partial\calD} \leq \vp.$ The choice of $C$ depends only on $\|\vp_1-\vp_0\|_{C^0(D)},\|\vp\|_{C^{0,1}([0,1]\times\partial D)}$. Then Lemma~\ref{lm:subs} shows $v_i$ is a subsolution of the $\calF_c$-Dirichlet problem for $(\calD,\vp).$ By inclusion~\eqref{SAThm} and the maximum principle~\eqref{eq:mp}, we deduce that $u \geq v_i$ on $\cal D.$ From the definition of $v_i,$ it follows that
\[
u(t,x) - \vp_0(x) \geq -Ct,  \qquad u(t,x) - \vp_1(x) \geq -C(1-t), \qquad (t,x) \in \calD.
\]
This proves the desired lower bounds on $u(t,x).$ To obtain the analogous upper bounds, we apply Lemma~\ref{lm:subs} to $-\vp_i$ to obtain subsolutions $v_i$
(different from the $v_i$ of the previous paragraph) to the $(\calF_{-c},-\vp)$ Dirichlet problem. The bounds on $C$ are the same as before. Again using inclusion~\eqref{SAThm} and maximum principle~\eqref{eq:mp}, we obtain $-u \geq v_i$ on $\calD.$ So, from the definition of $v_i,$ we conclude that
\[
u(t,x) - \vp_0(x) \leq Ct, \qquad u(t,x) - \vp_1(x) \leq C(1-t), \qquad (t,x) \in \calD,
\]
as desired.
\end{proof}
\subsection{Partial Lipschitz estimate}
\label{LipschitzSubSec}

\begin{lemma}\label{tm:Lip}
Suppose $u$ solves the $\calF_c$-Dirichlet problem for $(\calD,\vp)$ and
\[
\vp|_{[0,1] \times\partial D} \in C^{0,1}, \qquad \vp|_{\{i\}\times D} \in
F_{c-\pi/2}\cap \left(-F_{-c-\pi/2}\right), \quad i = 0,1.
\]
Then for each $x\in D$,  $u(\,\cdot\,,x)\in C^{0,1}([0,1])$. Moreover,
\[
\sup_{x\in D}||u(\,\cdot\,,x)||_{C^{0,1}([0,1])}\le C
=
C(||\vp_0||_{C^{0}},||\vp_1||_{C^{0}},||\vp||_{C^{0,1}([0,1]\times\del D)},D).
\]
\end{lemma}
\begin{proof}
We treat the forward Lipschitz bound. The proof of the backward Lipschitz bound is similar.
Fix $t_0 \in (0,1)$ and let
\beq\label{vuCInEq}
v(t,x)=u_{}(t_0,x)-C(t-t_0).
\eeq
By Lemma~\ref{lm:C0ts} and the assumption that $\vp|_{[0,1]\times D} \in C^{0,1}$, we can choose $C$ large enough so that
\beq\label{vuInEq}
 \h{$v\le u$   on    $\del([t_0,1]\times D)$.}
\eeq
Lemma~\ref{lm:rest} implies $u|_{\{t_0\}\times D} \in F_{c-\pi/2}\cap \left(-F_{-c-\pi/2}\right).$ So, Lemma~\ref{lm:subs} implies $v\in \calF_c([t_0,1]\times D),$ and inclusion~\eqref{SAThm} gives $v-u \in \SA([t_0,1]\times D).$ By \eqref{vuInEq} and the maximum principle~\eqref{eq:mp}, it follows that $v \leq u$ on $[t_0,1]\times D.$ Thus, since $v(x,t_0) = u(x,t_0)$ and $v$ is Lipschitz in $t$ by construction, we deduce that $u$ is Lipschitz in $t$ from below at $t_0$ with Lipschitz constant $C.$ Replacing $u$ by $-u$ and $\calF_c$ by $\calF_{-c}$, we get Lipschitz in $t$ from above with the same Lipschitz constant. For $t_0 = 0,$ we use the assumption on $\vp_0 = u_0$ in place of Lemma~\ref{lm:rest}.
\end{proof}

\section{The
%Riemannian and
calibration measure}
\label{CalibSec}

The goal of this section is to show that the calibration $\Re\O$
has a well-defined restriction to the Lagrangian $\graph(df)\subset \CC^n$, interpreted in a suitable weak sense, for any $f \in F_c$  (resp. $f \in -F_{-c}$) in the top (resp. bottom) branches. Here, by the top (resp. bottom) branches, we mean
\beq
\label{coutermostEq}
c\in I^n_{\topp}:=[(n-1)\pi/2,n\pi/2)
\h{\ \ \  (resp. $c\in I^n_{\bott}:=(-n\pi/2,-(n-1)\pi/2]$).}
\eeq
We call this restriction the {\it calibration measure}.
This is a priori non-trivial since the tangent space to the Lagrangian
need not exist everywhere.
% (cf. \cite[Proposition 2.17]{RZIII}).
The advantage of working in the outermost branches is that then
$F_c\subset \P$ or $-F_{-c}\subset -\P$. That is, our functions are convex/concave.
Intuitively, say, in the case $c \in I^n_{\topp},$ all but one of the eigenvalues of a matrix in $F_c$ must
be large, while the remaining eigenvalue is positive.
The following basic result is essentially a corollary of
the fundamental work of Rauch--Taylor  \cite{RauchTaylor}.
Let $X\subset \RR^n$ be a domain.
Denote by $M_p(X)$ differential $p$-forms on $X$ whose coefficients
are Borel measures on $X$, and endow $M_p(X)$ with the topology of weak convergence
of measures. Let $\dx:=dx^1\wedge \cdots \w dx^n$.

\begin{theorem}
\label{CalibThm}
\mbox{}
\begin{enumerate}
\item\label{it:ext}
For $\theta \in (-\pi,\pi],$ the map $C_\th: \pm\P(X)\cap C^2(X)\ra M_n(X)$ defined by
$$
C_\th(f): =\Re \big(e^{-\i \th} \det (I + \i\nabla ^2f)\big)\dx,
$$
admits a unique continuous extension to $\pm\P(X)$.
More precisely, if $f_i\in \pm\P(X)\cap C^2(X)$ converges to $f\in \pm\P(X)$ in the $C^0$ topology on compact subsets of $X,$
then $C_\th(f_i)$ converges weakly to $C_\th(f)$.
\item\label{it:calpos}
Let $c\in I^n_{\topp}$ (resp. $c \in I^n_{\bott}$) be such that $c = \theta-\pi/2 +2\pi l$  (resp. $c = \theta + \pi/2 + 2\pi l$) for some $l\in\Z$.
If $f \in F_c(X)$ (resp. $-f \in F_{-c}(X)$), then $C_\th(f)$ is a positive measure.
\end{enumerate}
\end{theorem}
\begin{remark}
The measure $C_\theta(f)$ is the restriction of $\Re\Omega$ to $\Lambda_f := \graph(df) \subset \C^n$ in the following sense. Let $g : \RR^n \to \Lambda_f$ be given by $g(x) = x + \sqrt{-1}df(x).$ Then by formula~\eqref{CndzEq} we have
\[
g^*\Re\left(e^{-\i\th}\Omega\right) = C_\theta(f).
\]
\end{remark}
The proof of Theorem~\ref{CalibThm} is given at the end of this section. It relies on the following result of Rauch--Taylor \cite[Proposition 3.1]{RauchTaylor}.
\bprop\label{pr:RT}
Let $I:=\{i_1,\ldots,i_q\}$ denote a set of integers with
 $1\le i_1<\cdots< i_q\le n$. Then the operator $\calM_I:C^2(X)\cap\P(X)
\ra M_q(X)$ defined by
$$
\calM_I(u) = d\frac{\del u}{\del x_{i_1}}\w\cdots
\w d\frac{\del u}{\del x_{i_q}}
$$
admits a unique continuous extension to an operator
defined on all of $\calP(X)$ with respect to the
$C^0(X)$ topology on $\calP(X)$ and the weak convergence of measures
on $M_q(X)$.
\eprop
We also need the following lemmas.
\begin{lemma}\label{lm:contP}
If $c \in I^n_{\topp},$ then $F_c \subset \Int\P.$
\end{lemma}
\begin{proof}
By definition, $A \in F_c$ iff $\tr\tan^{-1}(A) \geq c.$ In particular, $c \in I^n_{\topp}$ implies that $A$ is positive definite.
\end{proof}
\begin{lemma}\label{Fcconv}
If $c \in I^n_{\topp},$ then $F_c \subset \Symn$ is a convex subset.
\end{lemma}
\begin{proof}
Recall that $\tan^{-1}$ is concave on the positive real axis. So,
it follows from Lemma~\ref{lm:contP} that
if $A,B \in F_c$ and $t \in [0,1],$ then
\cite[Corollary 1]{Davis}
\[
\tr\tan^{-1}(tA + (1-t)B) \geq t \,\tr\tan^{-1}(A) + (1-t) \tr\tan^{-1}(B) \geq t c + (1-t) c  = c.
\]
Thus $tA + (1-t)B \in F_c$ as desired.
\end{proof}

\begin{lemma}\label{lm:FcXconv}
If $c \in I^n_{\topp}$, then $F_c(X)$ is convex.
\end{lemma}
\begin{proof}
By Lemma~\ref{lm:contP}, we have $F_c(X) \subset \P(X).$ So, $F_c(X)$ consists of convex functions. Let $f_0,f_1 \in F_c(X).$ By Alexandrov's theorem, $f_i$ is a.e. twice differentiable. By property~\ref{it:hess} we have $\nabla^2f_i(x) \in F_c$ for almost every $x.$ For $t \in [0,1],$ Proposition~\ref{Fcconv} implies
\[
t \nabla^2 f_0(x) + (1-t) \nabla^2 f_1(x) \in F_c\cap \P= F_c
\]
for almost every $x.$ So $t f_0 + (1-t) f_1 \in F_c(X)$ by Theorem~\ref{tm:ae}.
\end{proof}
In the following, for $i \in \Z_{>0}$ we write $X_i = \{x \in X\,:\,\dist(x,D) > i^{-1}\}.$
\begin{lemma}\label{lm:approx}
Let $F \subset \Symn$ be a subequation such that $F(X_i)$ is convex for $i \in \Z_{>0}.$ Then for each $f \in C^0(X) \cap F(X)$ there exists a sequence $f_i \in C^\infty(X_i)\cap F(X_i)$ that converges uniformly to $f$ on every compact subset of $X.$
\end{lemma}
\begin{proof}
Let $\eta : \RR^n \to \RR$ be smooth with support in the unit ball at the origin and $\int \eta = 1.$ For $i \in \Z_{>0},$ let $\eta_i: \RR^n \to \RR$ be given by $\eta_i(x) = i^n \eta(i x).$ Consider $f \in C^0(X) \cap F(X).$ The sequence $f_i \in C^\infty(X_i)$ given by
\[
f_i(x)  = \eta_i * f(x) = \int_{X} \eta_i(y) f(x-y) dy
\]
converges uniformly to $f$ on every compact subset of $X.$ It remains to show that $f_i \in F(X_i).$ Indeed, recall that a partition $P$ of a cell $I = \prod_{i=1}^n [a_i,b_i]$ is a collection of (not necessarily closed) cells such that $I = \coprod_{J \in P} J.$ A collection of midpoints $Y$ for a partition $P$ is a collection of elements $y_J \in J$ for each $J \in P.$ The content of $I$ is denoted by $c(I) = \prod_{i=1}^n (b_i-a_i).$
On any compact subset of $X,$ the function $f$ is uniformly continuous. Moreover, $\eta$ is Lipschitz. So, there exists a sequence of partitions $P_{ij}$ of the hypercube $[-1/i,1/i]^n$ such that for any collections of midpoints $Y_{ij},$ the associated Riemann sums
\[
S(P_{ij},Y_{ij})(x) = \sum_{J \in P_{ij}} \eta(y_J)f(x- y_J) c(J)
\]
approximate $f_i(x)$ uniformly in $x$ on compact subsets of $X_i.$
Moreover, we can choose the sets of midpoints $Y_{ij}$ for $P_{ij}$ such that
\[
\sum_{J \in P_{ij}} \eta_i(y_J)c(J) = 1,
\]
so the Riemann sum $S(P_{ij},Y_{ij})$ is a convex combination of functions in $F(X_i).$ Since $F(X_i)$ is convex by assumption, we deduce that $S(P_{ij},Y_{ij}) \in F(X_i).$ The lemma follows by property~\ref{it:uniflimit}.
\end{proof}

\begin{proof}[Proof of Theorem~\ref{CalibThm}]
Denote by $\sigma_k(A)$ the elementary symmetric  polynomial of degree $k$ in the eigenvalues of $A$.
First, for $f \in C^2(X),$ we have
\begin{align*}
C_\th(f)
&=
\cos\th \sum^{\lfloor n/2\rfloor}_{k=0}(-1)^{k}\sigma_{2k}(\nabla ^2f)\dx
+ \sin\th \sum_{k = 0}^{\lfloor (n-1)/2\rfloor}(-1)^k \sigma_{2k+1}(\nabla^2 f) \dx.
\end{align*}
A detailed explanation of how to prove this identity can be found in~\cite[p. 91]{HL82}. Since
\[
C_{-\theta}(-f) = C_{\theta}(f),
\]
it suffices to prove the theorem in the case $f \in F_c(X)$ with $c \in I^n_{\topp}.$
For each $k=0,\ldots,n$, we claim the map $C^2(X)\ni f\mapsto \sigma _k (\nabla ^2f)\dx\in M_n(X)$ admits a unique
continuous extension to a map from $\P(X)$ to $M_n(X)$.
For that, recall that $\sigma_k(A)$ can be written as a sum of (up to sign) determinants of principal
$k$-by-$k$ sub-matrices of $A$. Thus, the claim, and hence also part~\ref{it:ext} of the theorem, follows from Proposition~\ref{pr:RT}.

It remains to prove part~\ref{it:calpos}. We assume first that $f \in C^2(X) \cap F_c(X)$ and prove $C_\theta(f)$ is a positive measure. Indeed, recall the definition of the Lagrangian angle $\theta_f$ and its lift $\tilde\theta_f$ from~\eqref{eq:LA} and~\eqref{liftedSLangleEq}. Since $f \in F_c(X),$ we have $\tilde \theta_f \geq c.$ On the other hand, $\tilde\theta_f = \tr\tan^{-1}(\nabla^2f) < n\pi/2.$ So, the assumption on $c$ implies that $\theta - \pi/2 + 2\pi l\leq \tilde\theta_f < \theta + 2\pi l.$ In particular, $\theta_f \in [\theta- \pi/2,\theta) \subset S^1.$ Thus
\[
C_\theta(f) =\Re \left(e^{-\i (\th - \theta_f)} \left|\det (I + \i\nabla ^2f)\right|\right)\dx = \cos(\theta - \theta_f)\left|\det (I + \i\nabla ^2f)\right| \geq 0,
\]
as desired.

Finally, consider the case of general $f \in F_c(X).$ By Lemmas~\ref{lm:FcXconv} and~\ref{lm:approx}, there exists a sequence of functions $f_i \in C^2(X_i)\cap F_c(X_i)$ converging to $f$ in the $C^0$ topology on compact subsets of $X.$ By part~\ref{it:ext} of the theorem, $C_\th(f) = \lim_{i\to\infty} C_\th(f_i).$ Being a limit in the weak topology of positive measures, $C_\th(f)$ must be positive.
\end{proof}

\section{The length and calibration functionals}
\label{CalFunctionalSec}
\subsection{Length of weak solutions of
the geodesic equation}\label{ssec:length}

Combining Theorems \ref{tm:cgeq} and \ref{CalibThm},
Corollary \ref{calFcConvCor}
and Lemma \ref{tm:Lip},
%and \ref{weakcalOLemma}
we obtain that
the length of the weak geodesics produced
in this article is well-defined in the outermost branches.

\bthm
\label{lengthThm}
Let $c \in I_{\topp}^{n+1}$
(recall (\ref{coutermostEq})) and let $\vp\in
C^0(\del\calD)$
be a consistent function such that
$$
\vp_i = \vp|_{\{i\}\times D}\in F_{c-\pi/2}\cap(-F_{-c-\pi/2})=F_{c-\pi/2}, \q i=0,1,
$$
and $\vp|_{[0,1]\times \del D}$ is affine in $t.$
Let $k$ be the solution of the $\calF_c$-Dirichlet
problem for $(\calD,\vp)$ given by Theorem \ref{tm:cgeq}. Let $\theta \in (-\pi,\pi]$ satisfy $\theta + 2\pi l = c$ for $l \in \Z.$ Then the length integral
$$
\int_0^1\left(\sqrt{\int_D (\dot k(t,x))^2C_\theta(k_t)(x)}\right)dt
$$
is well-defined.
\ethm

\subsection{The calibration functional}\label{ssec:cal}

We now briefly return to the setting of
 Lagrangians in a general
Calabi--Yau manifold $X$ as in \S\ref{SPLSubSec}.
Denote by $\calO$ an orbit of $\ham(X,\omega)$ acting on the space $\calL$ of oriented Lagrangian submanifolds of $X$ diffeomorphic to a manifold $L.$ The Lagrangian submanifolds in $\calO$ need not be positive. If $L$ is compact, denote by $d \in H_n(X)$ the homology class represented by the Lagrangian submanifolds in $\calO.$ The following is a restatement of~\cite[Theorem 1.1]{S1}

\def\Ga{\Gamma}

\bthm
\label{CalibFnctlThm}
Let $\beta$ be a closed $n$-form on $X$ such that $\omega \wedge \beta = 0$ and if $L$ is compact, $\int_d \beta = 0.$ Let $\La:[0,1]\ra\calO$ denote
a smooth path in $\calO$. Then the integral
$$
\calC_\beta(\La)=\int_0^1\left(\int_{\La_t} \frac{d\La_t}{dt}\beta|_{\Lambda_t}\right) dt
$$
depends only on the homotopy class of $\La$ relative to its endpoints.
\ethm
By type analysis, we have $\omega \wedge \Omega = 0,$ so we can take $\beta = \Re \Omega$ or $\Im\Omega$ or a linear combination thereof. Let $\calO_\th \subset \calO \cap \calL_\th^+$ be a connected component and consider $\beta = \Im\big(e^{-\i\th}\O\big).$ Holding $\Lambda_0$ fixed and varying $\Lambda_1,$ we obtain from $\calC_\beta(\Lambda)$ a functional on the universal cover $\widetilde \calO_\th,$ which is shown in~\cite[Sec. 5]{S1} to be convex with respect to the metric $(\cdot,\cdot)_\th$ and to have critical points over special Lagrangians. In particular, if any two points in $\calO_\th$ can be connected by a geodesic of $(\cdot,\cdot)_\th,$ then $\calO_\th$ contains at most one special Lagrangian. This is one of the main reasons for studying the DSL equation.

In the following, consider instead $\beta = \Re\big(e^{-\i\th}\O\big).$ For this choice of $\beta,$ abbreviate $\calC = \calC_\beta.$ In the compact case, we cannot impose $\int_d \beta = 0,$ because the Lagrangians in $\calO_\th,$ which represent $d,$ are positive. So,  $\calC$ is not well defined in the compact case. However, it suffices for the purposes of the following discussion to restrict to $L$ non-compact.

We call $\calC$ the {\it calibration functional}
since its Fr\'echet differential at $\La$ holding $\Lambda_0$ fixed
is precisely the restriction of
the calibration $\Re\big(e^{-\i\th}\O\big)$ to $\La_1$ considered as a linear function on $T_{\Lambda_1}\calO \simeq C_0^\infty(\Lambda_1).$ See~\cite[Prop. 3.3]{S1}. It is thus natural to restrict $\Lambda_1$ to belong to a connected component $\calO_\th \subset \calO \cap \calL_\th^+.$
Informally, $\calC$ can be thought of as the function on
$\calO_\th$ whose gradient is the vector field with value the constant function $1$
on each Lagrangian $\Gamma \in \calO_\th,$ or as the potential
for the 1-form defined by the calibration measure. For a brief overview of an analogous functional that appears in the context of the complex \MA equation, see the beginning of Section 4 in~\cite{R} and references therein.

The importance of the functional $\calC$ is
that it is linear along smooth geodesics.

\bthm
\label{LinearSmoothThm}
Let $\Lambda : [0,1]^2 \to \calO$ be a family of paths and write $\Lambda_{t,\tau} = \Lambda(t,\tau).$ Suppose the path $t \mapsto \Lambda_{t,0}$ is constant, and the path $t \mapsto \Lambda_{t,1}$ is a geodesic in $(\mathcal O_\th,(\cdot,\cdot)).$
Then
\[
\frac{d^2}{dt^2} \calC\left(\La|_{\{t\}\times[0,1]}\right)=0.
\]
\ethm

\bpf
By~\cite[Prop. 3.3]{S1} we have
$$
\frac{d}{dt} \calC(\La|_{\{t\}\times[0,1]})=\int_{\La_{t,1}} \frac{d\La_{t,1}}{dt}\Re\big(e^{-\i\th}\O\big).
$$
Before differentiating once more in $t$,
we rewrite this integral using
the family of diffeomorphisms
$\tilde g_t:L \ra\La_{t,1}$
from \S\ref{LeviCivitaSubSec} that
satisfies
\begin{equation*}
\iota_{d\tilde g_t/dt}\Re\big(e^{-\i\th}\O\big)=0.
\end{equation*}
Then
\begin{equation}\label{eq:dtreOm}
\frac{d}{dt}\tilde g_t^*\Re\big(e^{-\i\th}\O\big)
=d \iota_{d\tilde g_t/dt}\Re\big(e^{-\i\th}\O\big)=0
\end{equation}
for all $t\in[0,1]$. Writing $h_t = \frac{d\Lambda_{t,1}}{dt},$ we have
$$
\frac{d}{dt} \calC(\La|_{\{t\}\times[0,1]})
=
\int_{L}
(h_t\circ \tilde g_t) \tilde g_t^*\Re\big(e^{-\i\th}\O\big).
$$
Therefore, differentiating in $t$ gives
\[
\frac{d^2}{dt^2} \calC(\La|_{\{t\}\times[0,1]})
=
\int_{L}
\del_t(h_t\circ \tilde g_t) \tilde g_t^*\Re\big(e^{-\i\th}\O\big) + \int_{L}
(h_t\circ \tilde g_t) \del_t\left(\tilde g_t^*\Re\big(e^{-\i\th}\O\big)\right) =0
\]
by~\eqref{ReformulateGeodEq}, the geodesic equation $\frac{Dh_t}{dt} = 0,$ and~\eqref{eq:dtreOm}.
\epf

We now return to the setting of graph Lagrangians
in $X=\CC^n$. Restricting to paths $\Lambda$ in graphs of differentials of convex functions, we find that $\calC(\Lambda)$ depends only on $\Lambda_0$ and $\Lambda_1.$ We discuss the expected behavior of $\calC$ along Harvey--Lawson solutions of DSL. Parallel statements in the concave case also hold but will be omitted.

Fix $v \in \P(\RR^n)\cap C^\infty(\RR^n)$ and denote by
$$
\P_v(\RR^n)
$$
the set of all $u \in \P(\RR^n)$ such that $u - v$ has compact support. For $u \in \P_v(\RR^n),$ define $k \in C^0([0,1]\times \R^n)$ by $k(t,x) = tu(x) + (1-t)v(x)$ and write $k_t(x) = k(t,x).$ So, $k$ is differentiable in $t,$ the $t$ derivative $\dot k$ has compact support, and $k_t$ is convex for all $t.$ Let $\theta \in (-\pi,\pi].$ Then by Theorem~\ref{CalibThm}, the functional
\[
\calC(v,u):= -\int_0^1\left(\int_{D}\partial_t k(t,x) \,
C_\theta(k_t)\right)dt
\]
is well defined.

Take $\calO$ to be the orbit of the Lagrangian $\graph(v)\subset \C^n$ under $\ham(\C^n,\omega).$ If $u \in C^\infty(\RR^n),$ then also $k \in C^\infty([0,1]\times \RR^n).$ So, we have a path $\Lambda: [0,1] \to \calO$ given by $\Lambda_t = \graph(dk_t) \subset \C^n$. The calculations of \S\ref{sec:graphs} show that $\calC(u,v)=\calC(\La)$.

Motivated by results in pluripotential theory
for the \MA operator \cite[Remark~4.5]{BermanBoucksom}, it is natural
to make the following conjecture characterizing
weak geodesics. Let $c \in I_{\topp}^{n+1}$ %\cup I_{\bott}^{n+1}$
(recall (\ref{coutermostEq})) and
let $u\in\calF_c((0,1)\times\RR^n)$ be such that $u_t - v$ has compact support for $t \in (0,1).$ In particular, by Lemmas~\ref{lm:rest} and~\ref{lm:contP}, we have $u_t \in \P_v(\RR^n)$ for $t\in (0,1).$

\bconj
The function
$t\mapsto \calC(v,u_t)$ is affine in $t$
if and only if $-u\in \widetilde\calF_c((0,1)\times\RR^n)$, that is, $u$ solves the DSL in the sense of \HL.
\econj
\begin{remark}
It is not immediately clear how to formulate the preceding conjecture if we consider a bounded domain $D \subset \RR^n$ instead of $\RR^n.$ Indeed, even if $u,v \in C^\infty(D),$ satisfy $u|_{\del D} = v|_{\del D},$ it may not be the case that $du_x = dv_x$ for $x \in \del D.$ Thus the boundaries of $\graph(du)$ and $\graph(dv)$ need not coincide. The calibration functional does not in general behave well on families of Lagrangians that do not agree at their boundaries.
\end{remark}

\appendix
\section{Limiting eigenvalues and a formula for
the lifted space-time Lagrangian angle}
\lb{LimitDSet}

The purpose of this subsection is to derive
an explicit formula for the lifted space-time angle
$\wtThk$. The result is stated
in Corollary \ref{wtThkFormulaCor}.
This formula is not needed for any of our main results,
but we feel it is of independent interest. Also, it furnishes
alternative proofs of some of the results in Section
\ref{DDSec}.

Let $I^a_n:=\diag(a,1,\ldots,1)\in\Sym(\R^{n+1})$. It is natural to approximate the DSL \eqref{DSLMainEq} by the following
family of equations parametrized by $p>0$:
\beq\lb{GraphDSLApproxEq}
\Im\big( e^{-\i \th}\det(I^{1/p^2}_n+\i\nabla^2 k)\big)=0.
\eeq
This can be rewritten as
\beq\lb{GraphDSLApprox2Eq}
\Im\big(e^{-\i \th}\det(I+\i I^{p}_n\nabla^2 k I^{p}_n)\big)=0.
\eeq
The reason for rewriting the equation
in this manner is that the matrix version of the special Lagrangian (SL) equation on $\R^{n+1}$,
\beq\lb{GraphSLEq}
\Im\big(e^{-\i \th}\det(I+\i A)\big)=0,
\eeq
is equivalent to the equation \cite[p. 438]{HL}
\beq\lb{GraphSL2Eq}
\sum_{i=0}^{n}\tan^{-1}\la_i(A)=c+k\pi, \q k\in\Z,\; |k|<(n+1)/2,
\eeq
where $\tan^{-1}:\RR\ra(-\pi/2,\pi/2)$, and
 $\{\la_i(A)\}_{i=0}^n$ are the (real) eigenvalues of the symmetric $(n+1)$-by-$(n+1)$ matrix
$A$ ordered so that $\la_0(A) \geq \cdots \ge \la_{n}(A)$.
We would like to express the DSL in a similar manner.
To that end, we analyze the limiting behavior of the eigenvalues
 of $I^{p}_n\nabla^2 k I^{p}_n$ as $p$ tends to infinity, or more generally
of
\beq\lb{ApEq}
A_p:=I^{p}_n A I^{p}_n,
\eeq
for any $A\in\Sym(\R^{n+1})$.
%We then use this analysis to associate subequations with \eqref{DSLEq}.

To state the result concerning the eigenvalues, we introduce the following
notation. Given a matrix $A=[a_{ij}]_{i,j=0}^n\in\Sym(\R^{n+1})$, we denote
the characteristic polynomial of $A$ by
\[
\chi_A(\lambda):= \sum_{i = 0}^{n+1} (-\lambda)^i \sigma_{n+1-i}(A),
\]
where $\sigma_j$ denotes the sum of all principal $j$-by-$j$ minors of $A$, or equivalently,
the $j$-th symmetric polynomial in the eigenvalues of $A$. We use the convention
that $\sigma_0(A)=1$ and $\sigma_j(A)=0$ if $j>n+1$.

Let $A^+ \in \Sym(\R^n)$ be defined by
\beq
%\label{AplusEq}
A^+:= [a_{ij}]_{i,j=1}^n.
\eeq
%Whenever $A=\diag(0,A^+)$, set $\chi_A^\infty(\mu):=\chi_{A^+}(\mu)$.
For $A\in\Sym(\RR^{n+1})$ define,
\[
\rho_i(A):=
%\begin{cases}
%\sigma_i(A^+),& \text{if $A=\diag(0,A^+)$,}\\
\sigma_i(A) - \sigma_i(A^+),
\qq i=0,\ldots,n,
 %& \text{otherwise,} \\
%\end{cases}
\]
set $\rho_{n+1}(A):=\sigma_{n+1}(A)$ and set
\beq\label{chiinftyEq}
\chi_A^\infty(\mu):= \sum_{i=0}^n (-\mu)^i \rho_{n+1-i}(A).
\eeq
Note that $\chi_A^\infty$ has degree $n$ precisely
when $a_{00}\not=0$. % or when $A=\diag(0,A^+)$.
When $a_{00}=0$ but $A\not=\diag(0,A^+)$, the polynomial
$\chi_A^\infty$ has degree $n-1$. Indeed,
\begin{equation}\label{eq:rhot}
\rho_2(A) = -\sum_{i=1}^{n}a_{0i}^2 <0.
\end{equation}
According to Lemma \ref{ApproxEigValLemma} below,
regardless of the degree of $\chi_A^\infty$, all of its
roots are real. Therefore, the following definition makes
sense.

\bdefin
\label{muDef}
Let $\chi_A^\infty$ be defined by
\eqref{chiinftyEq}.
Denote the roots of $\chi_A^\infty$   by
\beq\label{muiEq}
\mu_1(A) \geq \cdots \ge \mu_n(A)
\eeq
when its degree equals $n$, and by $\mu_1(A) \geq \cdots \ge \mu_{n-1}(A)$ when the degree equals $n-1$.  In this latter case we set $\mu_{n}(A):=0$ for later convenience.
Finally, if $A=\diag(0,A^+)$, we set
\beq
\label{diagmuEq}
\mu_i(A):=\lambda_i(A^+), \q i=1,\ldots,n,
\eeq
which are, of course, real as well.
\edefin

\begin{lemma}
\lb{ApproxEigValLemma}
If $a_{00} > 0,$ then
$
\lim_{p \to \infty} \lambda_{i}(A_p) = \mu_i(A),\, i=1,\ldots,n,
\;
\lim_{p \to \infty} p^{-2}\lambda_0(A_p) = a_{00}.
$
When $a_{00} < 0,$ we have
$
\lim_{p \to \infty} \lambda_{i-1}(A_p) = \mu_i(A), \, i=1,\ldots,n,
\;
\lim_{p \to \infty} p^{-2}\lambda_{n}(A_p) = a_{00}.
$
When $a_{00} = 0$, but $A\not=\diag(0,A^+)$, we have
$
\lim_{p \to \infty} \lambda_{i}(A_p)=\mu_i(A), \, i=1,\ldots,n-1,
$
while $\lim_{p \to \infty}p^{-1}\lambda_{0}(A_p)=\sqrt{-\rho_2(A)}>0$
and $\lim_{p \to \infty}p^{-1}\lambda_n(A_p)=-\sqrt{-\rho_2(A)}<0$.
\end{lemma}

\def\s{\sigma}
\bpf
Our assumptions imply that
%We may suppose that
$A\not=\diag(0,A^+)$.
%, as  otherwise the result holds trivially.
At first, we also assume that $a_{00}\not=0$.

The polynomial $\s_j(A_p)$ is the weighted sum of all symmetric $j$-by-$j$ minors of $A$,
with weight equal to $p^2$ if the minor involves the first row and column of $A$,
and weight equal to $1$ otherwise.
In other words, $\s_j(A_p)= p^2\rho_j(A)+\s_j(A^+)$.
Thus,
\beq\lb{chiApEq}
\chi_{A_p}(\la) = \sum_{i=0}^{n+1} (-\la)^i \big[p^2\rho_{n+1-i}(A)+\s_{n+1-i}(A^+)\big].
\eeq
By definition, $\rho_0(A)=0$.
Hence, the equation $\chi_{A_p}(\la) = 0$ can be rewritten as
\beq\lb{chiAp3Eq}
%\sum_{i=0}^{n} (-\la)^i \rho_{n+1-i}(A)
\chi^\infty_A(\la)-p^{-2}q(\la)=0,
\eeq
where $q(\la)$ is a polynomial of degree $n+1$ whose coefficients
are bounded independently of $p$.
Since $A_p$ is symmetric, equation \eqref{chiAp3Eq} has $n+1$ real roots.
Recall that $\mu_1(A),\ldots,\mu_n(A)\in\C$ denote the $n$ roots of $\chi^\infty_A$.
Let $R>0$ be such that $|\mu_i(A)|<R-1$ for each $i$.
Applying the argument principle to the left hand side of \eqref{chiAp3Eq},
it follows that for all sufficiently large $p$,
$\chi_{A_p}(\la)$ has exactly $n$ roots in $\{z\in\C\,:\, |z|<R\}$.
Moreover, $\{\mu_i(A)\}_{i=1}^n$ is the limit set of these $n$ roots.

On the other hand, the sum of the eigenvalues of $A_p$ equals $\tr A_p =  p^2a_{00}+\tr A^+,$ and we have already showed that $n$ of the eigenvalues of $A_p$ are bounded. It follows that $p^2a_{00}$ is an eigenvalue of $A_p$ up to $O(1)$. This completes the proof in the case $a_{00}\not=0$.

Suppose now that $a_{00}=0$ but still $A\not=\diag(0,A^+)$.
Then $\chi^\infty_A$ has degree $n-1$ and the argument principle still
implies that $n-1$ of the eigenvalues of $A_p$ limit to $\mu_i(A), i=1,\ldots,n-1$.
On the one hand, $\sum \la_i(A_p) = \tr A_p = (-1)^n\tr A^+$, implying that the sum of the remaining two eigenvalues of $A_p$ is
bounded independently of $p$.
On the other hand, $\sigma_2(A_p) = (p^2\rho_2(A)+\sigma_2(A^+))=O(p^2)\in\RR$. This polynomial also equals
$\sum_{i>j}\la_i(A_p)\la_j(A_p)$. Putting all these facts together and keeping in mind~\eqref{eq:rhot}, we obtain
\[
\la_n(A_p)=-\sqrt{-\rho_2(A)}p+o(p), \qquad \la_0(A_p)=\sqrt{-\rho_2(A)}p+o(p),
\]
with
$\la_n(A_p)+\la_0(A_p)=O(1)$.
\epf

Set
%by $\sign : \R \to \{-1,0,1\}$ the function given by
\[
\sign(x):= \begin{cases}
-1, & x < 0, \\
0, & x = 0, \\
1, & x > 0.
\end{cases}
\]

\bthm
\lb{uscargFormulaThm}
For $A=[a_{ij}]_{i,j=0}^n\in\Sym(\RR^{n+1})$ with $A \neq \diag(0,A^+),$ let $\mu_1(A),\ldots,\mu_n(A)$
be as in Definition \ref{muDef}. Let $\widehat\Theta(A)$ be given by formula~\eqref{eq:whTh}.
Then
\[
\widehat \Theta(A)= \lim_{p \to \infty} \tr\tan^{-1}(A_p) =
\frac\pi2\sign(a_{00})+
\sum_{j=1}^n\tan^{-1}\mu_j(A).
\]
\ethm

As an immediate corollary we obtain a formula for the
space-time Lagrangian angle.

\bcor
\label{wtThkFormulaCor}
The lifted space-time Lagrangian angle is given by
$$
\wtThk(t,x)=
\begin{cases}
\displaystyle
\frac\pi2\sign(\ddot k(t,x))+
\sum_{j=1}^n\tan^{-1}\mu_j(\nabla^2k(t,x)),
&
\hbox{if $\nabla^2k(t,x)\not=\diag(0,\nabla^2_xk(t,x))$,}
\cr
\displaystyle
\frac\pi2+\tr\tan^{-1}(\nabla^2_xk(t,x)),
&
\h{otherwise.}
\end{cases}
$$
\ecor

\begin{proof}[Proof of Theorem~\ref{uscargFormulaThm}]
Denote by $\arg: \C \to (-\pi,\pi]$ the argument function.
For $B \in \Sym(\CC^{n+1})$ denote by $\spec(B)$ the spectrum of $B,$ and let
\[
\mathcal B = \{B \in \Sym(\CC^{n+1})| \spec(B) \cap \RR_{\leq 0} = \emptyset\}.
\]
Define $\arg : \Sym(\CC^{n+1})\setminus\mathcal B \to \Sym(\CC^{n+1})$ by
\[
\arg(B):=\frac1{2\pi\i}
\int_\gamma (\zeta I-B)^{-1}\arg\zeta\, d\zeta
\]
where $\gamma$ is a contour in $\CC\setminus \RR_{\leq 0}$ enclosing $\spec B.$ It follows from the definition that $\arg(B)$ depends continuously on $B \in \Sym(\CC^{n+1})\setminus\mathcal B.$
By \cite[p. 45]{Kato} the eigenvalues of $\arg(B)$ are the arguments of the eigenvalues of $B$ with corresponding multiplicities, so
\begin{equation}\label{eq:argdettrarg}
\arg \det B=\tr\arg(B)\;\mod 2\pi.
\end{equation}
By Lemma~\ref{RealPartEVLemma} we have $\spec(I^{\eps^2}_n+\i A)\cap\RR_{\leq 0} = \emptyset$ for $\eps \geq 0.$ Moreover, for $\eps>0$,
\begin{equation}\label{eq:argeps}
\baeq
\arg\det(I^{\eps^2}_n+\i A)
&=
\arg\det
\big(I_n^{\eps}(I+\i A_{1/\eps})I_n^{\eps}\big)
\cr
&=
\arg \eps^2\det(I+\i A_{1/\eps})
=
\arg \det(I+\i A_{1/\eps}).
\eaeq
\end{equation}
Combining equations~\eqref{eq:argdettrarg} and~\eqref{eq:argeps}, we obtain
\begin{equation}\label{eq:trargeps}
\tr\arg(I+\i A_{1/\eps})=
\tr\arg(I^{\eps^2}_n+\i A)
\;\mod 2\pi.
\end{equation}
Since the left and right hand sides of congruence~\eqref{eq:trargeps} are continuous functions of
$\eps$ that coincide for $\eps=1$, it follows that they are actually
equal.
Setting $\eps=1/p$ and using the continuity of $\arg$ gives
\begin{align*}
\widehat\Theta(A) = \tr\arg(I_n+\i A)&=\lim_{p\ra\infty}\tr\arg(I^{1/p^2}_n + \i A)\\
       & =  \lim_{p\ra\infty}\tr\arg(I+\i A_{p}) = \lim_{p\ra\infty}\tr\tan^{-1}(A_p).
\end{align*}
Theorem \ref{uscargFormulaThm} now follows
directly from Lemma \ref{ApproxEigValLemma}.
Indeed, this is clear in the case $a_{00}\not=0$.
In the case $a_{00}=0$ this follows as well,
since the smallest and largest eigenvalues of $A_p$
have asymptotically canceling contributions to $\tr\tan^{-1}(A_p).$
\end{proof}

\bremark
An alternative proof of Lemma \ref{AlmostDSetLemma}
can be given by using the results of the Appendix.
To this end, recall the definition of the subequation $F_c \subset \Symnplusone$ from~\eqref{eq:Fc}. Let
$$
\calF^p_c:= \big\{A \in \Sym(\R^{n+1})\,:\,
A_p\in F_c\}.
$$
We first prove that $\calF^p_c$ is a subequation.
Indeed, let $A\in \calF^p_c$ and $P\in\P$.
Now $A_p\in F_c$, and for any $P\in\P$ also $P_p\in\P$.
So, since $F_c$ is a subequation, $A_p+P_p=(A+P)_p\in F_c$. Thus,
$A+P\in\calF_c^p$, and $\calF_c^p$ is a subequation.

Next, by Lemma \ref{ApproxEigValLemma}, given
$A\in\calF_c$ such that $A\neq\diag(0,A^+)$ and $\eps>0$, there
exists $p_0$ such that $A_p\in F_{c-\eps}$,
i.e., $A\in\calF^p_{c-\eps}$, for all $p\ge p_0$. Since $\calF_{c-\eps}^p$ is a subequation, given any $P\in\P$
one has
$A+P\in\calF^p_{c-\eps}$. By Lemma \ref{ApproxEigValLemma} there exists $p_1$ such
that if $p\ge\max\{p_0,p_1\}$ then $A+P\in\calF^p_{c-\eps}$ implies
that $A+P\in \calF_{c-2\eps}$ whenever $A+P\not=\diag(0,A^++P^+)$.
The implication continues to hold when $A+P=\diag(0,A^++P^+)$ because $\calF^p_{c-\eps}$ is closed (Lemma~\ref{lm:calFcc}) and the set of $P \in \P$ such that $A+P\not=\diag(0,A^++P^+)$ is dense.
Since $\eps>0$ was arbitrary, it follows that $A+P\in\calF_c$.

On the other hand, if $A \in \calF_c$ and $A = \diag(0,A^+)$, then $A + \delta I \in \calF_c$ for all $\delta > 0$ by Lemma~\ref{ApproxEigValLemma}. Moreover, $A + \delta I \neq \diag(0,A^+ + \delta I)$.
So, for all $P \in \P,$ we have $A + \delta I + P \in \calF_c.$ But $\calF_c$ is closed, so this implies $A + P \in \calF_c$ as desired.
\eremark

\section{Geometric formulation of the DSL}\label{sec:gdsl}
This section is devoted to a geometric formulation of the DSL equation valid in an arbitrary Calabi-Yau manifold $(X,J,\omega,\Omega).$ We use the notation of Section~\ref{SPLSec}. In particular, $n = \dim_\C X$ and $\calO_\th$ is an exact isotopy class of $\theta$-positive Lagrangians in $X.$

To $X$ we associate the complex $(n+1)$-dimensional manifold
\[
Y_X = [0,1]\times \R \times X.
\]
Let $p_X : Y_X \to X$ denote the projection, and denote by $x,y,$ the projections to $[0,1],\R,$ respectively. We think of $x,y,$ as coordinates, and denote by $\del_x,\del_y,$ the corresponding vector fields. We equip $Y_X$ with the family of complex structures $\{\tJ_\epsilon\}_{\epsilon > 0}$ defined by
\[
\tJ_\eps \del_x  = \eps \del_y, \qquad \tJ_\eps \del_y = \frac{1}{\eps} \del_x, \qquad \tJ_\eps \xi = J\xi \quad \text{for $\xi$ tangent to $X$}.
\]
Furthermore, we equip $Y_X$ with the K\"ahler form
\[
\tom = dx\wedge dy + p_X^*\omega
\]
and the $\tJ_\eps$-holomorphic $(n+1,0)$-form
\[
\tO_\eps = (\eps dx + \i dy) \wedge p_X^*\Omega.
\]
To a path $\Lambda : [0,1] \to \calO_\th,$ we associate the Lagrangian suspension
$
\Gamma_\Lambda \subset Y_X
$
defined as follows. Let $h_t : \Lambda_t \to \R$ be given by $h_t = d\Lambda/dt.$ Then
\[
\Gamma_\Lambda = \{(x,y,z)\in [0,1]\times\R\times X| z \in \Lambda_t, y = -h_x(z)\}.
\]
It is well-known that $\Gamma_\Lambda$ is a Lagrangian submanifold~\cite{Po93}. Observe that while $\tJ_\epsilon$ is not well-defined when $\epsilon = 0,$ the form $\tO_0 = \i dy \wedge p_X^*\Omega$ is well-defined.
\begin{prop}
The path $\Lambda$ is a geodesic of the metric $(\cdot,\cdot)_\th$ if and only if
\begin{equation}\label{eq:gdsl}
\Im  e^{-\i \th}\tO_0|_{\Gamma_\Lambda} = 0.
\end{equation}
\end{prop}
\begin{proof}
Choose a family of diffeomorphisms $g_t : L \to \Lambda_t$ and let $w_t$ be the vector field on $L$ defined as in~\eqref{eq:crfl}. Denote by $p_L : [0,1]\times L \to L$ and $t : [0,1]\times L \to [0,1]$ the canonical projections. Let $\tg : [0,1]\times L \to \Gamma_\Lambda$ be the diffeomorphism given by
\[
\tg(t,p) = (t,-h_t(p),g_t(p)).
\]
So, equation~\eqref{eq:gdsl} is equivalent to $\tg^*\im e^{-\i \th}\tO_0 = 0.$ But
\begin{multline}\label{eq:gh}
\tg^*\im e^{-\i \th}\tO_0  = \\
\begin{aligned}
&= \tg^*(dy \wedge \Re e^{-\i \th} \tO_0)\\
& = \big(-p_L^*d(h_t\circ g_t) - \partial_t (h_t \circ g_t) dt\big) \wedge \left(p_L^*g_t^*\Re e^{-\i \th}\Omega + dt \wedge p_L^* i_{dg_t/dt} \Re e^{-\i \th} \Omega\right )\\
& = -dt \wedge p_L^*\left(\partial_t (h_t\circ g_t) g_t^*\Re e^{-\i \th}\Omega - d(h_t \circ g_t)\wedge i_{dg_t/dt}\Re e^{-\i \th}\Omega\right).
\end{aligned}
\end{multline}
Moreover, by~\eqref{eq:crfl} and the signed derivation property of the interior product, we have
\begin{multline}\label{eq:sder}
d(h_t \circ g_t)\wedge i_{dg_t/dt}\Re e^{-\i \th}\Omega = \\ =-d(h_t \circ g_t)\wedge i_{w_t}g_t^*\Re e^{-\i \th}\Omega = - g_t^* dh_t (w_t)  g_t^*\Re e^{-\i \th}\Omega.
\end{multline}
Combining equations~\eqref{eq:gh} and~\eqref{eq:sder} and recalling \eqref{eq:lcc}, we obtain
\begin{multline*}
\tg^*  e^{-\i \th}\im \tO_0
= - \Big (\partial_t (h_t\circ g_t) + g_t^* dh_t (w_t)\Big) dt \wedge p^*_L g_t^* \Re e^{-\i \th} \Omega = \\
=- \left(\frac{Dh_t}{dt}\circ g_t\right) dt \wedge p^*_L g_t^* \Re e^{-\i \th} \Omega.
\end{multline*}
Since the Lagrangians $\Lambda_t$ are positive, the form $dt \wedge p^*_L g_t^* \Re e^{-\i \th} \Omega$ vanishes nowhere. The proposition follows.
\end{proof}
\begin{remark}
For $\epsilon > 0,$ the equation $\Im  e^{-\i \th}\tO_\eps|_{\Gamma_\Lambda} = 0$ is the special Lagrangian equation for $\Gamma_\Lambda$ with respect to $\tO_\eps.$ So, as in Appendix~\ref{LimitDSet}, the DSL equation arises as a limiting case of the non-degenerate elliptic special Lagrangian equation.
\end{remark}
\section*{Acknowledgments}
This work was supported by BSF grant 2012236.
YAR was also supported by NSF grant DMS-1206284 and a Sloan Research Fellowship.
JPS was also supported by ERC starting grant 337560.
We are grateful to M. Dellatorre and a referee for a careful reading.

\def\listing#1#2#3{{\sc #1}:\ {\it #2}, \ #3.}

\vspace{.5 cm}
\noindent

{\sc University of Maryland}

{\tt yanir@umd.edu}

\bigskip

{\sc Hebrew University}

{\tt jake@math.huji.ac.il}

\end{document}